\documentclass[a4paper,11pt,reqno]{amsart}
\pdfoutput=1
\usepackage[a4paper,margin=2.4cm]{geometry}
\usepackage[backref=page]{hyperref}
\usepackage{comment}
\usepackage{hyperref}
\usepackage{amssymb}
\usepackage{amsmath}
\usepackage{mathrsfs}
\usepackage{tikz-cd}
\usepackage{stmaryrd}

\usepackage{todonotes} 
\usepackage{xcolor,hyperref}
\definecolor{darkred}{rgb}{.8,0,0}
\definecolor{tocolor}{rgb}{.1,.1,.1}
\definecolor{urlcolor}{rgb}{.2,.2,.6}
\definecolor{linkcolor}{rgb}{.1,.1,.5}
\definecolor{citecolor}{rgb}{.4,.2,.1}
\definecolor{gray}{rgb}{.8,.8,.8}

\hypersetup{
	colorlinks=true,
	urlcolor=urlcolor,
	linkcolor=linkcolor,
	citecolor=citecolor,
	pdfauthor = {Vestislav Apostolov, Brent Pym and Jeffrey Streets},
	pdftitle = {Poisson K-stability and cscGK structures}
	}

\AtEndDocument{\vfill\eject\batchmode}

\newtheorem{iTheorem}{Theorem}
\newtheorem{iProp}{Proposition}
\newtheorem{iCor}{Corollary}
\newtheorem{iConj}{Conjecture}
\newtheorem{Theorem}{Theorem}[section]
\newtheorem{conjecture}[Theorem]{Conjecture}
\newtheorem{lemma}[Theorem]{Lemma}
\newtheorem{prop}[Theorem]{Proposition}
\newtheorem{cor}[Theorem]{Corollary}

\theoremstyle{definition}
\newtheorem{defn}[Theorem]{Definition}

\theoremstyle{remark}
\newtheorem{rem}[Theorem]{Remark}

\newtheorem{ex}[Theorem]{Example}

\newcounter{numl}
\newcommand{\labelnuml}{\textup{(\roman{numl})}}

\DeclareSymbolFont{script}{U}{eus}{m}{n}
\DeclareSymbolFontAlphabet{\amathscr}{script}
\DeclareMathSymbol{\Wedge}{0}{script}{"5E}
\DeclareMathAlphabet{\mathrmsl}{OT1}{cmr}{m}{sl}

\newcommand{\GK}{\mathcal{GK}}
\newcommand{\AGK}{\mathcal{AGK}}
\newcommand{\K}{\mathcal{K}}

\newcommand{\Ham}{{\rm Ham}}
\newcommand{\Alb}{{\rm Alb}}
\renewcommand{\bar}[1]{\overline{#1}}

\newcommand{\Goto}{{\rm Goto}}
\newcommand{\Gscal}{\mathrm{Gscal}}
\newcommand{\gD}{\Delta}
\newcommand{\IP}[1]{\left<#1 \right>}

\renewcommand{\geq}{\geqslant}
\renewcommand{\leq}{\leqslant}

\newcommand{\R}{{\mathbb R}}
\newcommand{\C}{{\mathbb C}}
\newcommand{\N}{{\mathbb N}}
\newcommand{\Q}{{\mathbb Q}}
\newcommand{\T}{{\mathbb T}}

\newcommand{\G}{{G}}

\newcommand{\PP}{{\mathbb P}}

\newcommand{\cO}{{\mathcal O}}

\newcommand{\tstL}{{\mathscr L}}
\newcommand{\tstA}{{\mathscr A}}

\newcommand{\tstX}{{\mathscr X}}
\newcommand{\tstF}{{\mathscr F}}

\newcommand{\Futc}{\mathrm{Fut}}
\newcommand{\DF}{\mathrm{DF}}
\newcommand{\poiss}{\sigma}
\newcommand{\Poiss}{\Sigma}

\newcommand{\tr}{{\rm tr}}

\newcommand{\Aut}{\mathrm{Aut}}

\newcommand{\del}{\partial}

\newcommand{\Scal}{\mathrm{Scal}}
\newcommand{\Sph}{{\mathbb S}}
\newcommand{\Z}{{\mathbb Z}}
\def\th/#1#2{{#1}^{#2}}

\renewcommand{\N}{\nabla}
\newcommand{\brs}[1]{\left|#1\right|}

\renewcommand{\i}{\sqrt{-1}}

\DeclareMathOperator{\Rm}{Rm}
\DeclareMathOperator{\Rc}{Rc}
\DeclareMathOperator{\divg}{div}
\DeclareMathOperator{\inj}{inj}

\begin{document}

\title{Poisson K-stability and the semiclassical Yau--Tian--Donaldson correspondence}

\author[V. Apostolov]{Vestislav Apostolov} \address{Vestislav Apostolov \\ D{\'e}partement de Math{\'e}matiques\\ UQAM\\ C.P. 8888 \\ Succursale Centre-ville \\ Montr{\'e}al (Qu{\'e}bec)  H3C 3P8 \\ Canada and Institute of Mathematics and Informatics, Bulgarian Academy of Sciences \\ Acad. Georgi Bonchev Str., Block 8 \\ 1113 Sofia, Bulgaria} 
\email{apostolov.vestislav@uqam.ca}

\author[B. Pym]{Brent Pym} 
\address{Brent Pym \\ Department of Mathematics and Statistics\\ McGill University \\
Burnside Hall
805 Sherbrooke Street West
Montreal (Quebec) H3A 0B9 \\ Canada}
\email{brent.pym@mcgill.ca}

\author[J. Streets]{Jeffrey Streets} \address{Jeffrey Streets \\ Department of Mathematics\\
         University of California, Irvine\\
         Irvine, CA 92617}
\email{jstreets@uci.edu}

\thanks{V.A. was supported in part by an NSERC Discovery Grant RGPIN-2023-03316 and a Team FRQNT Grant during the preparation of this work. B.P.~was supported by the Natural Sciences and Engineering Research Council of Canada (NSERC), through Discovery Grants RGPIN-2020-05191 and RGPIN-2026-07433. J.S. was supported by the NSF via DMS-2203536. We thank F.\ Bischoff, M.\ Gualtieri, Y.\ Odaka  and A.\ Lahdili for insightful discussions. }

\begin{abstract}  We introduce a notion of K-polystability for compact K\"ahler holomorphic Poisson manifolds.  On the one hand, this notion of stability is well-adapted to constructions of moduli spaces.  For instance, when the underlying manifold is K-polystable with reductive reduced automorphism group, Poisson K-stability is equivalent to geometric invariant theoretic stability in the space of Poisson bivectors, but there also exist K-unstable varieties that become stable after incorporating a Poisson structure.  On the other hand, the Poisson K-stability condition interacts well with generalized K\"ahler metrics---the background geometry of (2,2) supersymmetric string theory.  In particular, we conjecture that Poisson K-polystability characterizes the existence of constant scalar curvature symplectic generalized K\"ahler structures with a sufficiently small Poisson tensor---a natural extension of the Yau--Tian--Donaldson (YTD) conjecture.  Our main result is a proof of the existence part of this ``semiclassical YTD conjecture'' for Poisson structures on K\"ahler--Einstein Fano manifolds, using infinite-dimensional momentum map techniques.  In this way, we obtain the existence of many new examples of symplectic generalized K\"ahler structure of constant scalar curvature, and prove the conjecture completely in the case of the projective plane.


\end{abstract}

\maketitle
\setcounter{tocdepth}{1}
\tableofcontents

\section{Introduction}

\subsection{The Yau--Tian--Donaldson conjecture}
The theory of constant scalar curvature K\"ahler (cscK) metrics, arising from a visionary work of Calabi~\cite{calabi0} in the 1950s, has been a subject of active research for well over 70 years now.  This problem acquires rich structure through a formal momentum map interpretation discovered by Fujiki~\cite{fujiki} and Donaldson~\cite{donaldson-moment}.  The central conjecture in the field,
still open in its full generality, is the \emph{Yau--Tian--Donaldson} (YTD) conjecture, which states, broadly speaking, that for a compact K\"ahler manifold $X$ and a fixed class $\alpha \in H^{1,1}_{\mathrm dR}(X, \R)$, the existence of a cscK metric with class $\alpha$ is equivalent to an algebro-geometric notion of stability  for the pair $(X, \alpha)$; see, e.g., \cite{do-02,Tian,yau}. 

In a spectacular breakthrough, the YTD conjecture was proven for any smooth Fano variety $X$ endowed with its anti-canonical class $\alpha=2\pi c_1(X)$.  In this case, a cscK metric is necessarily K\"ahler--Einstein. There are, by now, many different proofs of this result: see \cite{BBJ21, CDS1, CDS2, CDS3, CSW18, ChiLi, Tia15,  Zha23} for the `if' part  and \cite{Be, Tian} for the `only if' part of the correspondence. This led to deep applications to the structure of moduli spaces of K\"ahler--Einstein (or equivalently, K-polystable) smooth Fano varieties~\cite{LWX, odaka1}. 

\subsection{$(2,2)$-supersymmetry beyond the K\"ahler case}
K\"ahler manifolds play an important role in string theory, providing examples of background geometries with $(2,2)$-supersymmetry.  In this realm, cscK metrics are linked with fixed points of the renormalization group flow.  However, not all $(2,2)$-supersymmetric string theory backgrounds arise from K\"ahler manifolds.  The relevant background geometry in general was described by Gates, Hull and Ro\v{c}ek in \cite{GHR}, under the name ``bihermitian geometry''.  These geometric structures were later rediscovered, and given the name of \emph{generalized K\"ahler} (GK) structures, in the context of Hitchin’s generalized geometry
program \cite{Gualtieri-PhD,Gualtieri-CMP,Hitchin-GCY}. In the ensuing decades, it has become clear that GK geometry is a deeply structured extension of K\"ahler geometry with novel implications for complex, symplectic and Poisson geometry.  Our goal in this paper is to give evidence for a YTD correspondence for generalized K\"ahler manifolds, which links the existence of constant scalar curvature generalized K\"ahler (cscGK) metrics with a suitable stability condition for holomorphic Poisson manifolds.

\subsection{Symplectic generalized K\"ahler manifolds}

The GK structures studied in this paper are those of ``symplectic type'', meaning that they are determined by a pair $(F,J)$, where $F$ is a symplectic form and $J$ is an integrable almost complex structure tamed by $J$, such that the tensor $I:= -F^{-1}J^* F$ is also an integrable almost complex structure.  K\"ahler manifolds correspond to the special case in which $F$ and $J$ are compatible, so that $I=J$, but in general, the complex manifolds $X=(M,J)$ and $X_- = (M,I)$ need not even be biholomorphic.

The complex structures $I$ and $J$ share a compatible Riemannian metric $g$, given by the symmetric part of $-FJ$, or equivalently, that of $-FI$.  Furthermore, by a result of Hitchin \cite{Hitchin0}, the complex manifolds $X$ and $X_-$ carry holomorphic Poisson structures with the same real part.  These holomorphic Poisson structures are equal to zero in the K\"ahler case, but can be highly nontrivial in general.  And while they need not be holomorphically isomorphic, they are related by a weaker notion of Morita equivalence~\cite{bischoff2018morita}, meaning that their spaces of symplectic leaves are isomorphic as 1-shifted holomorphic symplectic stacks.  

As explained in \cite{ASU,bischoff2018morita}, the complex manifold $X$, with its holomorphic Poisson structure, should be considered the ``holomorphic backdrop'' necessary to define a GK metric of symplectic type: once one fixes these data, and a de Rham class for the symplectic form $F$, a symplectic type GK structure is, roughly speaking, determined by a single smooth function, generalizing the classical K\"ahler potential.  More precisely, the connected components of this space are swept out by time-dependent Hamiltonian flows for the underlying real Poisson structure.  This results in a nonlinear structure and possibly nontrivial topology for the space of symplectic GK structures, a delicate new feature beyond the K\"ahler setting.  

\subsection{Scalar curvature and the string effective action}
The \emph{generalized scalar curvature} of a symplectic-type GK structure $(F,J)$ is the function
\begin{equation*}
\Gscal{(F, J)} := \Scal_g  -\tfrac{1}{12}|db|_g^2 + 2\Delta_g f - |df|^2_g, \qquad  f := \log\left(\frac{dV_F}{dV_g}\right), \end{equation*}
where $\Scal_g$ denotes the scalar curvature of the associated Riemannian metric $g=-(FJ)^{\rm sym}$, while $b:= -(FJ)^{\rm skew}$ is the ``torsion $2$-form'', $\Delta_g = d^*d$ is the corresponding Laplace operator, and $dV_F$  and $dV_g$ are the volume forms associated to  $F$ and $g$, respectively.  This expression is natural from several viewpoints:
\begin{itemize}
    \item It is the Lagrangian density for the string effective action~\cite{callan1985strings}.
    \item It is the natural notion of weighted scalar curvature associated to the cubic twisted Dirac operator \cite{Agricola, Bismut}.
    \item As shown in \cite{ASU}, it coincides with a momentum map on the space of almost generalized K\"ahler structures introduced by Boulanger~\cite{boulanger} in the toric case and Goto~\cite{Goto-moment} in general.
\end{itemize}
 This momentum map interpretation motivated, in \cite{ASU}, a natural extension of the Calabi program of finding cscK metrics in a given K\"ahler class.  Namely, fixing a complex manifold $X = (M,J)$, a holomorphic Poisson structure $\poiss$ on $X$, and a de Rham class $\alpha \in H^2(M, \R)$, we denote by
 \[
 \GK_{\poiss,\alpha} \subset \Omega^2(M,\mathbb \R)
 \]
 the set of symplectic forms $F$ with $[F]=\alpha$ such that $(F,J)$ is a GK structure with induced holomorphic Poisson structure $\poiss$.  We then seek an element $F \in \GK_{\poiss,\alpha}$ for which the generalized scalar curvature $\Gscal(F,J)$ is constant; we refer to $(F,J)$ as a constant scalar curvature generalized K\"ahler (cscGK) structure.

 Similarly to the cscK case, the formal momentum map interpretation of this problem yielded, in \cite{ASU}, a Mabuchi $1$-form on $\GK_{\poiss,\alpha}$, which vanishes at the cscGK metrics; a Futaki invariant defined on the Lie algebra of the group of reduced complex Poisson automorphisms of $(X, \poiss)$, which obstructs the existence of  a cscGK metric in $\GK_{\poiss, \alpha}$; and an additional Calabi--Lichnerowicz--Matsushima obstruction~\cite{ASU,Goto-Math.Ann}, which states  that if a cscGK metric exists, then the reduced automorphism group of $(X,\poiss)$ is reductive.

 \subsection{Poisson K-stability}
 By analogy with the cscK case, one would expect  that  the full obstruction  to the existence of cscGK structures in $\GK_{\poiss, \alpha}$ will be expressed  as a stability condition defined in terms of \emph{test configurations} of $(X, \poiss, \alpha)$. Strong evidence of such a link comes from \cite[Thm.1.8]{ASU} where,  in the case of a smooth \emph{toric} Poisson variety  $(X, \poiss, \alpha)$, it is shown that if $\GK_{\poiss, \alpha}$ contains a cscGK structure, then the toric K\"ahler manifold $(X, \alpha)$
is uniformly K-stable on \emph{toric} test configurations in the sense of \cite{do-02}.
Conversely, on a uniformly K-stable toric K\"ahler manifold $(X, \alpha)$  endowed with a toric Poisson tensor $\poiss$, there exists a constant $\epsilon>0$ such that  for each $|\lambda|<\epsilon$,  $\GK_{\lambda\poiss, \alpha}$ admits a cscGK metric. The proof of this result, however,  strongly relies on the toric assumption and the description  of the geometric structures on $X$ in terms  of data on the corresponding Delzant polytope, making it hard to generalize beyond the toric setting.  Our first main objective in this paper is therefore to introduce a candidate for a general complex analytic notion of  \emph{Poisson K-stability} associated to a smooth compact K\"ahler manifold $X$ endowed with a holomorphic Poisson tensor $\poiss$ and a K\"ahler deRham class $\alpha$. It is obtained as a natural extension of the notion of K-stability to the Poisson setting, as follows.

Recall from \cite{dervan-ross,s-dyrefelt1} that if $X$ is a compact complex manifold, and $\alpha \in H^{1,1}(X,\mathbb R)$ is a K\"ahler class, a \emph{K\"ahler test configuration for $(X,\alpha)$} (sometimes also called a \emph{cohomological test configuration}) is a compact normal complex variety $\tstX$ endowed with a K\"ahler de Rham class $\tstA \in H^{1,1}(\tstX, \R)$, a holomorphic $\C^*$-action, a $\C^*$-equivariant flat morphism $\tstX\to\PP^1$ that is equivariantly trivial with fibres $\tstX|_t \cong X$ for $t \neq 0 \in \PP^1$. Thus, a test configuration defines a generically isotrivial degeneration of $(X,\alpha)$ to a possible singular analytic variety $X_0 := \tstX|_0$.  Such a test configuration has an associated numerical invariant, called the \emph{Donaldson--Futaki invariant}
\[
\DF(\tstX,\tstA) \in \mathbb{R}
\]
defined by a suitable intersection formula ~\cite{do-02,odaka1,Tian,wang}.  The K-stability of $(X,\alpha)$ is then expressed in terms of the positivity properties of the invariants $\DF(\tstX,\tstA)$ on suitable classes of test configurations.  We recall the details in \autoref{sec:stability} below.

If $\poiss$ is a holomorphic Poisson structure on $X$, and $\tstX$ is a test configuration, then by continuity, there is at most one way to extend the Poisson structure $\poiss$ on $X\cong \tstX|_\infty$ to a $\C^*$-invariant Poisson structure $\Sigma$ on $\tstX$ that is tangent to the fibres of the projection $\tstX\to\PP^1$.  If such an extension exists, we say that the test configuration $\tstX$ is $\poiss$-Poisson.  If, in addition, the central fibre $(X_0, \poiss_0)$ of $(\tstX,\Poiss)$ is isomorphic to $(X,\poiss)$, we say that $\tstX$ is a Poisson product test configuration. 
We propose the following definition:
\begin{defn}
    A holomorphic Poisson K\"ahler manifold $(X, \poiss, \alpha)$ is \emph{Poisson K-semistable} if for any $\poiss$-Poisson K\"ahler test configuration $(\tstX, \Poiss, \tstA)$,  we have ${\rm DF}(\tstX, \tstA) \geq 0$.  It is  \emph{Poisson K-polystable} if, furthermore, equality holds only for product $\poiss$-Poisson test configurations.
\end{defn}

Observe that this definition is loosely analogous to the notion of stable Higgs bundles, where the role of the vector bundle is played by the pair $(X,\alpha)$, and the role of the Higgs field is played by the Poisson tensor $\poiss$.  A fundamental observation is that if the underlying K\"ahler manifold is K-polystable, then the K-polystability of the triple is determined by geometric invariant theory (GIT) applied to the Poisson bivector:
\begin{iProp}[see Proposition~\ref{p:Kahler/Poisson-K-polystable}] \label{p:polystable-intro}
Suppose that $(X,  \alpha)$ is a K-polystable K\"ahler manifold whose reduced automorphism group $\Aut_{\rm red}(X)$ is reductive, and that $\poiss$ is a holomorphic Poisson structure on $X$.  Then $(X, \poiss, \alpha)$ is Poisson K-polystable if and only if $\poiss$ is a polystable point for the linear action of $\Aut_{\rm red}(X)$ on the space $H^0(X, \Wedge^2 T^{1,0}_X)$.
\end{iProp}

On the other hand, it can happen that a K-unstable pair $(X,\alpha)$ becomes Poisson K-polystable after incorporating a Poisson structure, at least if we restrict our attention to smooth test configurations. An example is provided by the second Hirzebruch surface, viewed as a compactification of the cotangent bundle of $\PP^1$, with the Poisson structure induced by standard symplectic structure on the cotangent bundle; see Example \ref{p:F_2-poisson-stability}.

An interesting feature of the notion of Poisson K-polystability is that, at least under some cohomological and smoothness assumptions, it can be transferred between the two different holomorphic Poisson manifolds $(X,\poiss)$ and $(X_-,\poiss_-)$  underlying a GK structure; see Lemma \ref{l:GK-test}.  Indeed, the example of the second Hirzebruch surface mentioned above exploits this correspondence by relating it to the K-polystable manifold $\PP^1\times \PP^1$ via a GK structure.  In this way, our proof of Poisson K-polystability for the second Hirzebruch surface, which is a purely complex geometric statement, makes essential use of generalized K\"ahler geometry.

\subsection{A semiclassical YTD conjecture}
\medskip
With this notion of Poisson K-polystability in hand, we study a sort of ``semiclassical'' analogue of the YTD correspondence, for GK manifolds that are ``close'' to classical K\"ahler manifolds, i.e.~they are given by small Poisson deformations of K\"ahler pairs $(X,\alpha)$. As a guiding expectation, we formulate the following. \begin{iConj}\label{c:intro}  Let $(X, \sigma, \alpha)$ be a compact K\"ahler holomorphic Poisson manifold.  Then the triple $(X,\poiss,\alpha)$ is Poisson K-polystable if and only if there exists $\epsilon > 0$ such that for any $\lambda \in \C$ with $|\lambda| <\varepsilon$, the triple  $(X, \lambda\poiss, \alpha)$ admits a cscGK structure.
\end{iConj}
 We remark that as for the classical YTD conjecture, the notion of stability used in the present conjecture may need some modification outside the K\"ahler--Einstein case; see Remark \ref{rmk:caveats} below.  On the other hand, it suggests various natural extensions of this conjecture to the case of  ``large'' Poisson structures, or to more  general GK structures of non-symplectic type, by incorporating the formalism of \cite{AGJ,bischoff2018morita} based on Morita theory of symplectic groupoids; we leave such investigations for future work.

Our main result establishes one direction of the semiclassical YTD conjecture (existence of the metric) for Poisson structures on Fano K\"ahler--Einstein manifolds:
\begin{iTheorem}\label{thm:main} Suppose that $X$ is a K-polystable smooth Fano variety, and $\poiss$ is a holomorphic Poisson structure on $X$.  Then $(X, \poiss, 2\pi c_1(X))$ is Poisson K-polystable if and only if $\poiss$ is a polystable point for the linear action of $\Aut_{\circ}(X)$ on the vector space $H^0(X, \Wedge^2 T^{1,0}_X)$. In this case, there exists $\varepsilon >0$ such that  for all $\lambda\in \C$  with $|\lambda|<\varepsilon$, $X$ admits a cscGK symplectic GK structure $F_{\lambda} \in \alpha=2\pi c_1(X)$ with associated Poisson tensor $\lambda\poiss$.
\end{iTheorem}
In the special cases when the connected component $\Aut_{\circ}(X)$ of $\Aut(X)$ is trivial,  or $(X, \poiss)$ is toric, the existence part was established (using different methods) in \cite{Goto-Math.Ann} and \cite{ASU}, respectively, under the more general assumption that $(X, \alpha)$ admits a cscK metric. It is important to note, however, that these special solutions appear in cases when equivariant Poisson K-polystability  is equivalent to the equivariant K-polystability of $X$ itself, so that the notion of Poisson K-polystability presented here was not necessary.  Our present Theorem~\ref{thm:main} is new even for $X = \PP^2$, where we obtain, for example, the existence of a cscGK metric associated to a Poisson structure vanishing on a smooth cubic. Combined with the generalized K\"ahler Matsushima--Lichnerowicz obstruction~\cite{ASU, Goto-Math.Ann}, the existence part of Theorem~\ref{thm:main} actually confirms Conjecture~\ref{c:intro} completely for $\PP^2$.

More generally, we have the following:
\begin{iCor}\label{c:delPezzo}Suppose that $X$ is a K-polystable Del Pezzo complex surface and that $\sigma \in H^{0}(X, K^{-1}_X)$ is a Poisson structure.  Let $D\subset X$  be the anticanonical divisor on which $\poiss$ vanishes. Then $X$ admits a symplectic cscGK structure $F_{\lambda} \in 2\pi c_1(X)$ with Poisson tensor $\lambda \poiss$ (for $|\lambda|$  sufficiently small) in the following cases:
\begin{enumerate}
\item $X=\PP^2$ and $D$ is either a smooth cubic or a triangle of lines;
\item  $X=\PP^1 \times \PP^1$ and $D$ is isomorphic to one of the following: a smooth curve of genus one, the toric boundary divisor, or twice the diagonal.
\item $X= Bl_{p_1, p_2,p_3}(\PP^2)$ is the blow-up of $\PP^2$ at three points and $D$ is either smooth or a toric boundary.
\item $X$ is any other del Pezzo surface, and $\poiss$ is arbitrary.
\end{enumerate}
\end{iCor}

While these examples have dimension two, our methods readily produce higher dimensional cscGK metrics as well; for instance, we obtain a cscGK metric on $\PP^3$ associated with an elliptic normal curve via the celebrated Sklyanin Poisson structure; see Example \ref{ex:sklyanin}.

\subsection{Outline of the proof}
\bigskip The proof of Theorem~\ref{thm:main} uses the existence of a K\"ahler--Einstein metric $\omega_0 \in 2\pi c_1(X)$ provided by the Yau--Tian--Donaldson correspondence. This implies, in particular, that the reduced automorphism group of $X$ is reductive, so that Proposition~\ref{p:polystable-intro} gives the equivalence between Poisson K-polystability and GIT stability of the Poisson bivector. 
To prove the existence of a cscGK structure  in $\GK_{\lambda\poiss, 2\pi c_1(X)}$ in the polystable case, we  seek  to deform the K\"ahler--Einstein metric $\omega_0$ on $X$ to a cscGK structure $F_\lambda\in \GK_{\lambda\poiss, 2\pi c_1(X)}$ for small $|\lambda|$ using a perturbative analysis. To this end, we set up the necessary analytic background for applying the implicit function theorem, building on a construction of GK deformations of K\"ahler manifolds due to  Gualtieri~\cite{Gualtieri-Hamiltonian}. 

The primary difficulty is that if the K\"ahler--Einstein metric $\omega_0$ has a non-trivial connected isometry group $K$ (which is the case with some of the examples discussed in Corollary \ref{c:delPezzo} above), a direct application of the implicit function theorem is obstructed, because the Lie algebra of $K$  is identified with the kernel of the linearization of $\Gscal(F,J)$. If, furthermore, the connected Poisson automorphism group $\Aut_\circ(X, \poiss)$ is a strict subgroup of $\Aut_{\circ}(X)$, the latter obstruction cannot be resolved by working $K$-equivariantly, because a cscGK metric $F_\lambda\in \GK_{\lambda\poiss, 2\pi c_1(X)}$ cannot be $K$-invariant when $\lambda\neq 0$ by the Lichnerowicz--Matsushima obstruction~\cite{ASU}. 

To overcome this difficulty, we explain how to recast the setup as a finite-dimensional GIT problem, similar to techniques that have been  employed in K\"ahler geometry (see e.g. ~\cite{delloque, dms, ortu, sektnan-tipler, sz}) for solving different problems.  Namely, we relax the integrability hypotheses on the GK structure, to consider the space $\AGK_{\omega_0}$ of ``almost'' generalized K\"ahler structures which are tamed by a fixed symplectic form $\omega_0$, and we adapt a construction of generalized K\"ahler metrics from Poisson deformations due to Gualtieri \cite{Gualtieri-Hamiltonian} to the case of non-integrable bivectors, to obtain a $K$-equivariant ``slice'' map
\[
\tilde {\bf S} : H^0(X, \Wedge^2 T^{1,0}_X) \to \AGK_{\omega_0}
\]
Pulling back the formal momentum map picture for the generalized scalar curvature mentioned above, we reduce the problem to finding zeroes of the momentum map in the $\Aut_{\circ}(X)$ orbit of $\poiss$ for the linear $\Aut_{\circ}(X)$-action on $H^0(X, \Wedge^2 T^{1,0}_X)$, where the latter is endowed with the symplectic structure pulled back from $\AGK_{\omega_0}$.  The result then follows, in principle, from the Kempf--Ness theorem relating GIT stability to the structure of the momentum map.

However, there are two key technicalities that need to be overcome, leading to some additional results that may be of independent interest: 
\begin{enumerate}
\item In contrast with previous works in the K\"ahler case and {similarly to \cite{delloque}}, the slice map we construct need not be holomorphic.  Consequently, the symplectic structure on  $H^0(X, \Wedge^2T^{1,0}_X)$ need not be compatible with its natural complex structure, which is one of the hypotheses of the classical Kempf--Ness theory.  But in the K\"ahler--Einstein case, we are able to use curvature identities to show that the  symplectic form \emph{tames} the complex structure of $H^0(X, \Wedge^2T^{1,0}_X)$ in a sufficiently small neighbourhood of the origin.  

This {requires an extension} of the relevant (local) aspects of Kempf--Ness theory to the more general setting of linear actions on complex vector spaces with tame symplectic structures in a neighbourhood of the origin, see {\cite{delloque} and Section \ref{a:GIT}. The idea to use the dynamical stability of the momentum map flow in this setup is due to \cite{dms}.}
    
\item The construction of the slice invokes the implicit function theorem, which requires us to work in suitable H\"older spaces.   As a result, the cscGK structure we obtain from the momentum map picture is, a priori, only of finite regularity.  To deal with this issue, we prove an extension of the classical higher regularity theory for cscK structures to the GK setting (\autoref{t:schauder}).  In the classical K\"ahler case, the regularity can be obtained through use of the K\"ahler potential, or in a more geometric fashion via Bianchi identities.  GK metrics are again determined by local scalar potentials (cf. \cite{AGJ, GKP}), which is morally why the regularity persists in this setting.  But the theory of generalized K\"ahler potentials is quite subtle, so we instead use the geometric approach, roughly inspired by \cite{CW, TV}.  The analysis of the generalized scalar is considerably more intricate than in the classical case, but by a careful examination of identities for the Lee forms of generalized K\"ahler structures, we obtain an elliptic system that enables the desired bootstrapping.   
\end{enumerate}

\section{Symplectic generalized K\"ahler structures of constant scalar curvature}

We recall some fundamental theory of symplectic GK structures and the setup of \cite{ASU}.  Given this we recall the fundamental result of Gualtieri constructing symplectic GK structures via  deformation by a Poisson tensor.  We end the section by proving Theorem \ref{t:schauder}, an elliptic regularity result for symplectic constant scalar curvature GK structures.

\subsection{Generalized K\"ahler structures}

\begin{defn} \label{GKdef} Let $M$ be a smooth manifold, and let $H_0$ be a closed 3-form on $M$.  A \emph{generalized K\"ahler (GK) structure on $(M,H_0)$} is a tuple $(g, b, I, J)$, where $I$ and $J$ are
integrable almost complex structures, $b$ is a $2$-form, $g$ is a Riemannian metric compatible with both $I$ and $J$, and
furthermore
\begin{align*}
 d^c_I \omega_I = H_0 + db = - d^c_J \omega_J.
\end{align*}
\end{defn}
Every generalized K\"ahler structure $(g,b,I,J)$ has an associated Poisson structure, defined by the bivector field
\begin{align*}
 Q := \tfrac{1}{2} [I,J]g^{-1} \in \Wedge^2 TM.
\end{align*}
By \cite{apostolov1999bi,Hitchin0}, $Q$ is the real part of holomorphic $(2,0)$-bivectors
with respect to $I$ and $J$, which we denote by
\begin{gather} \label{e:PoissonJ}
\begin{split}
\poiss := Q - \i J Q, \qquad \poiss_- := Q - \i I Q.
\end{split}
\end{gather}
These bivector fields define holomorphic Poisson structures on the complex manifolds $X=(M,J)$ and $X_- = (M,I)$, respectively.

\subsection{GK structures of symplectic type}
In this paper, we focus on GK structures associated to symplectic forms on complex manifolds, extending the classical notion of compatibility in K\"ahler geometry, as follows.  Let $(M,J)$ be a complex manifold, and let $F$ be a symplectic form on $M$.  If the tensor
\begin{equation*}
 g := -(FJ)^{\rm sym}
\end{equation*}
is a Riemannian metric, and the tensor
\begin{equation*}
    I :=  -F^{-1} J^* F 
\end{equation*}
is an integrable almost complex structure, then the tuple $(g,b,I,J)$
defines a GK structure with torsion two-form
\begin{equation*}
b := -(FJ)^{\rm skew}
\end{equation*}
and underlying three form $H_0=0$.  Note that the pair $(F,J)$ determines the entire tuple $(g,b,I,J)$ algebraically.  For instance, the two complex structures and the real Poisson structure $Q=\tfrac{1}{2}[I,J]g^{-1}$ are related by
\begin{align}
I = J - QF. \label{eq:IJQF}
\end{align}

\begin{defn}
\label{symplecticGKdef} A generalized K\"ahler structure defined by a complex structure $J$ and a symplectic form $F$ as above is said to be of \emph{symplectic type}.  By abuse of notation, we will often refer to the pair $(F,J)$ as a \emph{symplectic GK structure}, or say that $F$ defines a \emph{symplectic GK structure on $(M,J)$.}
\end{defn}

Because of the nonlinearity of the conditions defining symplectic generalized K\"ahler structures, the notion of K\"ahler class is more subtle than in classical K\"ahler geometry.  As explained in \cite{ASU,bischoff2018morita}, the natural generalization of the classical K\"ahler class involves fixing not just the complex structure $J$ but also the associated holomorphic Poisson structure $\poiss$, as follows.

\begin{defn}\label{d:GK-class} 
\begin{enumerate}
\item Let $X = (M,J)$ be a complex manifold and let $\poiss$ be a holomorphic Poisson structure on $X$.  
A symplectic GK structure $F$ on  $X$ is called $\poiss$-\emph{compatible} if the $J$-holomorphic Poisson tensor associated to $(g, b, I, J)$ via (\ref{e:PoissonJ}) is equal to $\poiss$.  For a fixed holomorphic Poisson structure $\poiss$ on $X$, we denote by
\[
{\GK}_{\poiss} \subset \Omega^{2,cl}(M,\mathbb R)
\]
the space of $\poiss$-compatible symplectic GK structures on $X$.  

\item A de Rham class $\alpha \in H^2(X, \mathbb R)$ is \emph{$\poiss$-compatible} if it contains a $\poiss$-compatible symplectic GK structure, i.e.~$\GK_{\poiss} \cap \alpha \neq \varnothing \subset \Omega^{2,cl}(M,\mathbb R)$.  We denote by
\[
\GK_{\poiss,\alpha} \subset \GK_{\poiss}
\]
the space of symplectic GK structures $F$ on $(X,\sigma)$ such that $[F]=\alpha$. This space (or more precisely, a connected component thereof) is the appropriate generalization, for symplectic GK structures, of the K\"ahler class in classical K\"ahler geometry.
\end{enumerate}
\end{defn}

\begin{rem} \label{rem:Poisspreserved} By definition, a generalized K\"ahler class fixes the holomorphic Poisson structure $(J, \sigma)$.  Furthermore, it follows from \cite[Proposition 2.17]{ASU} that in a connected component of $\GK_{\poiss,\alpha}$, the second complex structure $I$ deforms via the action of the $\sigma$-Hamiltonian diffeomorphism group, hence the isomorphism class of the holomorphic Poisson structure $(I, \sigma_-)$ is preserved.
\end{rem}

\subsection{New symplectic GK structures from old}

In what follows we will make use of some useful constructions of generalized K\"ahler metrics.  The first construction is analogous to the classical statement that the restriction of a K\"ahler metric to a complex submanifold is again K\"ahler:
\begin{lemma}\label{lem:GK-restrict}
Let $(X,\poiss)$ be a holomorphic Poisson manifold, and let $F$ be a $\poiss$-compatible symplectic GK structure on $X$.  Suppose that $i : (Y,\eta) \hookrightarrow (X,\poiss)$ is a holomorphic Poisson immersion.  Then $i^*F$ is an $\eta$-compatible symplectic GK structure on $Y$, and $i$ induces a holomorphic Poisson immersion $i_{-} : (Y_-,\eta_-) \to (X_-,\poiss_-)$.
\end{lemma}

\begin{proof}
Let $M$ and $N$ be the real manifolds underlying $X$ and $Y$, respectively.
The problem is local, so we may assume without loss of generality that $N$ is an embedded submanifold of $M$, and identify $TN$ with a subbundle of $TM|_N$.  
Let $J$ be the integrable almost complex structure on $M$ defining the complex structure manifold $X$, let $I$ be the other complex structure of the GK pair, and let $Q$ be the associated real Poisson structure.  By hypothesis, $TN \subset TM|_N$ is preserved by $J$, and $Q|_N \in \wedge^2 TN$.  It follows from \eqref{eq:IJQF} that $TN$ is also preserved by $I$, giving a second complex structure on $N$.  Moreover, since $F$ tames both $I$ and $J$, and the restriction of a tame symplectic structure to a complex subspace is tame, we deduce that $i^*F$ tames both complex structures on $Y$, giving the desired GK structure.
\end{proof}

The second construction is much more subtle: it uses a holomorphic Poisson structure to deform a K\"ahler metric into a generalized K\"ahler one \cite{goto2021unobstructed,Gualtieri-Poisson,Gualtieri-Hamiltonian}; this is a special case of the more general stability theorem for generalized K\"ahler manifolds in \emph{op.~cit.} (see also \autoref{sec:gualtieri-map} below). The precise statement is as follows.

\begin{Theorem}\label{thm:gg-def}
    Let $(X,\poiss)$ be a compact holomorphic Poisson manifold with $H^{0,2}(X, \C)=0$, and let $\omega$ be a K\"ahler form on $X$.  Then there exists a constant $\epsilon > 0$ and a family of symplectic-type GK structures $F(t) \in \GK_{t\poiss}$ that depends analytically on $t \in (-\epsilon,\epsilon)$ and is invariant under the action of all biholomorphisms of $X$ that fix $\poiss$ and $\omega$.  Moreover, the Kodaira--Spencer class of the $F(t)$-conjugate family of complex manifolds $X_-(t)$ is given by
    \[
    \left.\frac{d}{dt}\right|_{t=0}[X_t] = [\poiss\omega] \in H^1(X, T^{1,0}_X)
    \]
\end{Theorem}
\begin{proof}
    The proof of this result is given in \cite{goto2021unobstructed,Gualtieri-Poisson,Gualtieri-Hamiltonian}, without mention of invariance under automorphisms.  For the latter, we simply observe that in \cite{Gualtieri-Hamiltonian}, an explicit Taylor expansion $F(t) = F_0 + t F_1 + \cdots$ is constructed inductively, with each term $F_k$ obtained from the previous ones by applying a map that depends functorially on the K\"ahler structure and the holomorphic Poisson structure.  Hence the construction is manifestly invariant under all symmetries of $(X,\poiss,\omega)$.
\end{proof}

The formula for the Kodaira--Spencer class in \autoref{thm:gg-def} is the analogue, for degree-one cohomology, of the map $ \cO_X \to T^{1,0}_X$ assigning to each function its Hamiltonian vector field with respect to the Poisson structure.  For this reason, it is given the following name in \cite{Gualtieri-Hamiltonian}:

\begin{defn}\label{def:ham-def}
    The family $(X_-(t),\poiss_-(t),\alpha)$ of Poisson K\"ahler manifolds induced by \autoref{thm:gg-def} is called the \emph{Hamiltonian deformation} of $(X,\poiss,\alpha)$.
\end{defn}

\subsection{Generalized scalar curvature}
Let $(F,J)$ be a symplectic GK structure on a manifold $X$ of dimension $2n$.  We denote by
\[
dV_F = \frac{F^n}{n!} \qquad dV_g := \sqrt{\det g}
\]
the volume forms associated to the symplectic structure $F$ and the Riemannian metric $g=-(FJ)^{\rm sym}$. (Here and throughout, we equip $M$ with the orientation defined by $F$, which determines the sign of the Riemannian volume form.)  Unlike in the K\"ahler case, the ratio of these two volume forms may be non-constant.  We denote its logarithm by
\begin{align}
f := \log\left(\frac{dV_F}{dV_g}\right). \label{eq:vol-ratio}
\end{align}
The function $f$, by definition of volume weight, is referred to in physics literature as the dilaton.

\begin{defn}
The \emph{generalized scalar curvature} of a symplectic GK structure $(F,J)$ is the function
\begin{equation}\label{eq:Gscal} \Gscal{(F, J)} := \Scal_g  -\tfrac{1}{12}|db|_g^2 + 2\Delta_g f - |df|^2_g, \qquad  f := \log\left(\frac{dV_F}{dV_g}\right), \end{equation}
where $\Scal_g$ denotes the scalar curvature of the associated Riemannian metric, $b:= -(FJ)^{\rm skew}$ is the torsion $2$-form, $\Delta_g = d^*d$ is the corresponding Laplace operator, and $f$ is defined as in \eqref{eq:vol-ratio}.
\end{defn}

\begin{defn}
A \emph{constant scalar curvature generalized K\"ahler} (or \emph{cscGK}) structure of symplectic type is a symplectic GK structure $(F,J)$ whose generalized scalar curvature $\Gscal(F, J)$, as defined in \eqref{eq:Gscal}, is constant.
\end{defn}

\subsection{Higher regularity of symplectic cscGK structures}

In this subsection we establish a bootstrapping result that can be used to upgrade the regularity of cscGK structures under natural geometric hypotheses.  This result will be useful in the proof of our main result, where we will first produce metrics that are, a priori, only finitely differentiable, and then apply the bootstrapping result to deduce that they are, in fact, smooth.  More precisely, we will prove the following result below.

\begin{Theorem} \label{t:schauder} Let Let $M^{2n}$ be a smooth manifold, and let $(F,J)$ be a symplectic cscGK structure on $M$ of regularity $C^{2 + \lambda}$ for some $\lambda > 0$.  Suppose that there exists a constant $\Lambda > 0$ such that
\begin{align*}
    \sup_M \{ \inj_g^{-2}, \brs{\Rm_g} \} < \Lambda.
\end{align*}
where $\inj_g$ denotes the injectivity radius of the metric $g = -(FJ)^{\rm sym}$ and $\Rm_g$ denotes the Riemann tensor.  Then $(F,J)$ is smooth, and for all $k \in \mathbb N$ there exists $C_k (n,\Lambda)> 0$ such that
\begin{align*}
    \sup_M \left( \brs{\nabla^k \Rm} + \brs{\nabla^{k+1} H} \right) \leq C_k.
\end{align*}
\end{Theorem}

\begin{cor}
If $M$ is compact smooth manifold, and $\lambda > 0$, then every symplectic cscGK structure on $M$ of regularity $C^{2+\lambda}$ is smooth.
\end{cor}

Before proving \autoref{t:schauder}, we recall some notation.  Let $(g,b,I,J)$ be a generalized K\"ahler structure of symplectic type.  We denote by   \[
\theta_I,\theta_J \in \Omega^1(M,\mathbb R)
\]
the associated Lee forms, defined as the trace of the torsion of the Chern connections for the Hermitian manifolds $(M,I)$ and $(M,J)$, respectively.
We denote by
\[
\rho_B \in \Omega^2(M,\mathbb R)
\]
the Ricci form of the Bismut connection, by
\[
\Rc^g \in S^2T^*M
\]
the Ricci tensor of $g$, and by
\[
H = db \in \Omega^3(M,\mathbb R)
\]
the exterior derivative of $b$.

The main point of the proof is to establish an elliptic system that includes the Ricci curvature.  These arguments are inspired by the fundamental aspects of related arguments \cite{CW,TV}, although we do not attempt here the sharper geometric regularity theorems shown therein.  In the cscK case this system is easily obtained through differential Bianchi identities.  In the cscGK case this strategy is complicated by the presence of torsion, as well as the weight function $f$ in the definition of scalar curvature.  However one still expects an elliptic system due to the existence of local generalized K\"ahler potentials.  Through careful differential identities we establish such an elliptic system in the next proposition.  In the statement we adopt the following ad-hoc notation: a term is denoted $O(T)$ if it is expressible as a linear map acting on a tensor $T$.  Furthermore, a term $Q^i$ is a tensor built from the generalized K\"ahler data containing terms which in coordinates contain no more than $i$ derivatives of the data $(g, I, J, b)$.  This is generic notation in the sense that the $Q^i$ terms differ from line to line.

\begin{prop} \label{p:cscGKelliptic} Let $(F,J)$ be a symplectic GK structure on a manifold $M$.  Then we have
\begin{align} \label{f:reg1}
    \gD (\theta_I^{\sharp} - \theta_J^{\sharp}) =&\ O(d \Gscal) + Q^2\\ \label{f:reg2}
    \gD f =&\ O(D (\theta_J - \theta_I)) + Q^1\\ \label{f:reg3}
    \gD \theta_I =&\ O(D^2(\theta_J - \theta_I) + Q^2\\ \label{f:reg4}
    \gD \rho_B =&\ O(D^3(\theta_I)) + O(D^4 f) + O(d \Gscal) + Q^3\\ \label{f:reg4.5}
    \gD H =&\ O(D \rho_B) + O(D^2(\theta_I - df)) + O(d \Gscal) + Q^2\\
    \label{f:reg5}
    \gD \Rc^g =&\ O(D^2 \rho_B) + O(D^3 H) + Q^3.
\end{align}
\begin{proof} To begin we work in the nondegenerate case, i.e. assuming $(I\pm J)$ are both non-degenerate and thus $(g, I, J)$ is symplectic with respect to either symplectic form $F_{\pm}=-2g(I\pm J)^{-1}$.  As we will derive universal tensor expressions, this suffices to treat the general case by the nondegenerate approximation method developed in \cite{AFSU}.  Recall here we set
\begin{align*}
    \Phi = \log \frac{F_+^n}{F_-^n},
\end{align*}
and it follows from \cite[Prop. 10.4]{Goto-moment} that
\begin{align} \label{f:reg10}
\Gscal = \tr_{F_+} d F_+ d \Phi.
\end{align}
Now recall \cite[Prop. 4.3]{apostolov2021the}, which implies that
\begin{align} \label{f:reg15}
    \theta_I^{\sharp} - \theta_J^{\sharp} = \sigma d \Phi,
\end{align}
and \cite[(3.4)]{ASU},  which says that
\begin{align*}
    \gD_{F _+} \Phi = \tr_{F_+} d F_+ g^{-1} d \Phi = \Delta \Phi - \tfrac{1}{2} \IP{d \log \det (I + J), d \Phi}.
\end{align*}
Using this identity and taking the Laplacian of equation (\ref{f:reg15}) proves (\ref{f:reg1}).

Next, we recall \cite[Proposition 4.3]{apostolov2021the} which implies that
\begin{align} \label{f:reg20}
    (I + J) d f = (I + J) d \log \det (I + J) = I \theta_I + J \theta_J = (I + J) \theta_I + J (\theta_J - \theta_I).
\end{align}
We also recall a general identity for pluriclosed structures (cf. \cite{ivanov2001vanishing}):
\begin{align} \label{f:reg30}
    d^* \theta_I =&\ \tfrac{1}{6} \brs{H}^2 - \brs{\theta_I}^2.
\end{align}
Applying $(I + J)^{-1}$ to (\ref{f:reg20}) and applying (\ref{f:reg30}) yields (\ref{f:reg2}).  Returning to (\ref{f:reg20}) we now express
\begin{align*}
    \theta_I = df - (I + J)^{-1} J (\theta_J - \theta_I),
\end{align*}
and (\ref{f:reg3}) follows from (\ref{f:reg1}) and (\ref{f:reg2}).

Turning now to the equation for $\rho_B$, we first recall the general identity (cf. \cite{ivanov2001vanishing})
\begin{align} \label{f:reg40}
    \rho_B(I, \cdot) = \Rc^H + \N \theta_I = \Rc^{H,f} + \N (\theta_I - df).
\end{align}
Taking the divergence of this equation, and using the twisted Bianchi identity
\begin{align*}
    \divg^{H,f} \Rc^{H,f} = \tfrac{1}{2} d R^{H,f} = \tfrac{1}{2} d \Gscal
\end{align*}
yields
\begin{align}
    d^* \rho_B = Q^2 + O(D^2(\theta_I - df)) + O(d \Gscal).
\end{align}
Recalling that $\rho_B$ is closed by the Bianchi identity and applying the Bochner formula then gives the claimed formula for $\rho_B$.

Now returning to (\ref{f:reg40}) we note it has the further consequence that
\begin{align*}
    d^* H = 2 \rho_B^{2,0 + 0,2}(I, \cdot) - d \theta_I + i_{\theta_I^{\sharp}} H.
\end{align*}
Using this, the fact that $H$ is closed, and the Bochner formula for three-forms gives (\ref{f:reg4.5}).  Returning now to equation (\ref{f:reg40}) we see that the equation (\ref{f:reg5}) for $\Rc^g$ follows formally.
\end{proof}
\end{prop}

With the above identities in place we sketch the standard bootstrapping argument to establish the desired higher regularity of cscGK structures.
\begin{proof}[Proof of \autoref{t:schauder}] We will first show the smoothness of the structure by an elliptic bootstrapping argument for the metric,  assuming it is already $C^{3+\lambda}$ in harmonic coordinates, as a direct consequence of Proposition \ref{p:cscGKelliptic}.  With this established, the theorem follows from a standard blowup argument.  So, with this assumption in place, using that $\Gscal$ is constant, equation (\ref{f:reg1}) implies that $\gD (\theta_I - \theta_J) \in C^{1+\lambda}$, hence $\theta_I - \theta_J \in C^{3+\lambda}$.  Then equation (\ref{f:reg2}) implies that $\gD f \in C^{2+\lambda}$, hence $f \in C^{4+\lambda}$.  Furthermore by equation (\ref{f:reg3}) we obtain that $\gD \theta_I \in C^{1+\lambda}$ hence $\theta_I \in C^{3+\lambda}$.  Then by equation (\ref{f:reg5}) we obtain that $\gD \Rc^g \in C^{\lambda}$, hence $\Rc^g \in C^{2+\lambda}$.  As we are in harmonic coordinates, it follows that $g \in C^{4+\lambda}$.  The bootstrap is complete and the argument above can be iterated to obtain that the metric is in fact $C^{\infty}$.
\end{proof}

\section{Test configurations and Poisson K-polystability}
\label{sec:k-stability}

In this section,  $X=(M, J)$ denotes a compact complex manifold and $\alpha \in H^{1,1}_{\rm dR}(X, \R)$ a deRham $(1,1)$-class containing a K\"ahler metric $\omega \in \alpha$. In the case when $\alpha = 2\pi c_1(L)$ for an ample holomorphic line bundle $L$, we say that $(X, L)$ is a smooth polarized variety.

\subsection{Automorphism groups}

We denote by $\Aut(X)$  the complex Lie group of holomorphic automorphisms of $X$.  It contains subgroups
\[
\Aut_{\rm red}(X) < \Aut_\circ(X) < \Aut(X,\alpha) < \Aut(X)
\]
where $\Aut(X,\alpha)$ is the subgroup fixing the K\"ahler class $\alpha$, while $\Aut_{\circ}(X)$ is the connected component of the identity, and $\Aut_{\rm red}(X)$ the group of reduced automorphisms, defined as the largest connected subgroup of $\Aut_\circ(X)$ that acts trivially on the Albanese torus 
\[
\Alb(X):=H^{1,0}(X, \C)^*/H_1(X,\Z).
\]
By work of Fujiki~\cite{Fujiki1978} and Lieberman~\cite{Lieberman1978}, $\Aut_{\rm red}(X)$ is an affine algebraic group.

The Lie algebra of ${\rm Lie}(\Aut_\circ(X))$ is identified with the space $H^0(X, T_X^{1,0})$ of holomorphic vector fields on $X$, and the ideal ${\rm Lie}(\Aut_{\rm red}(X)) < {\rm Lie}(\Aut_\circ(X))$ is identified with the subspace of vector fields whose vanishing set is nonempty; equivalently, they are the vector fields whose real part is Hamiltonian with respect to a K\"ahler form on $X$ (see \cite[Theorem 1]{lebrun1994deformation} and \cite{gauduchon-book}).

Moreover, by \cite[Theorem 3.5 and Corollary 3.7]{Fujiki1978}, the groups $\Aut_\circ(X,\alpha)$ and $\Aut_{\rm red}(X)$ have a common quotient, which is an affine algebraic group through which any compactifiable linear action of $\Aut_{\circ}(X)$ factors.  This applies, in particular, to the action of $\Aut_{\circ}(X)$ on the vector space $H^0(X,\mathcal{F})$ of global sections of any `natural' holomorphic bundle $\mathcal{F}$, such as the tangent bundle, forms, or tensors thereof; a compactification is obtained by viewing a section of the tensor bundle as a point in the Chow scheme of the compactification $\PP(\mathcal{F}\oplus \cO_X)$ of the total space of $\mathcal{F}$.  We will use this fact in the case $\mathcal{F} = \wedge^2 T^{1,0}_X$, so we record it as a lemma for future reference.
\begin{lemma}\label{lem:fujiki}
    The action of $\Aut_\circ(X)$ on $H^0(X, \wedge^\bullet T^{1,0}_X)$ factors through an algebraic action of a quotient of the affine algebraic group $\Aut_{\rm red}(X)$. Hence the $\Aut(X,\alpha)$-orbits in $H^0(X, \wedge^\bullet T^{1,0}_X)$ are locally closed subvarieties whose irreducible components are $\Aut_{\rm red}(X)$-orbits.
\end{lemma}
If $\poiss \in H^0(X, \wedge^2 T^{1,0}_X)$ is a Poisson structure we denote by
\[
\Aut(X,\poiss,\alpha) < \Aut(X,\alpha)
\]
the stabilizer of $\poiss$, i.e.~the set of biholomorphisms of $X$ that preserve both the K\"ahler class $\alpha$ and the Poisson structure $\poiss$.

 In the polarized case, we denote by $\Aut(X, L)$ the group of bundle automorphisms of $L$ and by $\Aut_{\circ}(X, L)$ its connected component of identity; then $\Aut_\circ(X, L)/\C^* \cong \Aut_{\rm red}(X)$.  


%
%
\subsection{Test configurations}
In this section we introduce a natural generalization of the notion of test configurations for K\"ahler manifolds \cite{Be,BDL,BHJ, H, ChiLi,s-dyrefelt2, Tian} to the Poisson setting.

\begin{defn}
    Let $(X,\poiss)$ be a compact holomorphic Poisson manifold.  An \emph{pre-test configuration for $(X,\poiss)$} 
    is a tuple consisting of the following data:
    \begin{itemize}
        \item A normal compact analytic space $\tstX$
        \item A flat morphism $\tstX\to \PP^1$
        \item A Poisson structure $\Poiss$ on $\tstX$ tangent to the fibres 
        \item A $\C^*$-action on $\tstX$ by Poisson automorphisms covering the standard action on $\PP^1$,
    \end{itemize}
  such that all fibres $(\tstX,\Poiss)|_t$ for $0\neq t \in \PP^1$ are isomorphic to $(X,\poiss)$, and the action of $\C^*$ on the fibre over infinity is trivial.

  A \emph{morphism of pre-test configurations $(\tstX,\Poiss) \to (\tstX',\Poiss')$} is a $\C^*$-equivariant Poisson morphism $\tstX \to \tstX'$ that projects to the identity map on $\PP^1$.
\end{defn}

\begin{defn}
    Let $(X,\poiss,\alpha)$ be a compact Poisson cohomological K\"ahler manifold.    A \emph{Poisson K\"ahler test configuration} $(\tstX,\Poiss,\tstA)$ is a pre-test configuration $(\tstX,\Poiss) \to \PP^1$ of $(X,\poiss)$ equipped with a K\"ahler class $\tstA \in H^{1,1}(\tstX,\R)$ such that $(\tstX,\Poiss,\tstA)|_\infty$ is isomorphic to $(X,\poiss,\alpha)$.
  
A \emph{morphism of Poisson test configurations $(\tstX,\Poiss,\tstA)\to(\tstX',\Poiss',\tstA')$} is a morphism of pre-test configurations $\phi :(\tstX,\Poiss)\to(\tstX',\Poiss')$ such that $\phi^*\tstA = \tstA'$.
\end{defn}

In the case $\poiss=0$, we shall tacitly identify the pair $(\tstX,0)$ with $\tstX$ and the triple $(\tstX,0,\tstA)$ with $(\tstX,\tstA)$.  In this way, we obtain the classical notion of (pre-)test configuration of the underlying K\"ahler manifold $(X,\alpha)$.

\begin{rem}
    In the definitions, we do not specify the isomorphisms of the fibres with $(X,\poiss,\alpha)$; we merely require their existence.
\end{rem}

\begin{rem}
In some references treating the classical case $\poiss=0$, one finds an \emph{a priori} stronger requirement that the family be $\C^*$-equivariantly trivial at infinity, but this is, in fact automatic from the stated conditions on the fibres.  This can be deduced by the same argument as in Lemma \ref{prop:subgroups} below, by interchanging the roles of $0,\infty\in \PP^1$.
\end{rem}

\begin{rem}
In the special case in which  the K\"ahler classes involved are integral (resp.~rational), they correspond to ample line bundles (resp.~$\Q$-line bundles), so that the spaces involved are projective algebraic varieties and we recover the more classical notions of test configurations involving line bundles.  Namely, if $(X,L)$ is a polarized variety equipped with the K\"ahler class with $\alpha = 2 \pi c_1(L)$, a \emph{polarized test configuration} of exponent $r \in \Z_{>0}$ is a K\"ahler test configuration $(\tstX,\tstA)$ equipped with a holomorphic line bundle $\tstL$ on $\tstX$ such that $\tstA = \frac{2 \pi}{r} c_1(\tstL)$. 
\end{rem}



A (pre-)test configuration can be thought of as an equivariant degeneration from $X \cong \tstX|_\infty$ (with the trivial action of $\C^*$), to $X_0 := \tstX|_0$ (with \emph{some} action of $\C^*$).  An important role is played by those (pre-test) configurations for which $X_0$ is also a copy of $X$, so that the degeneration is not really a degeneration at all.
\begin{defn}
    A Poisson K\"ahler test configuration $(\tstX,\Poiss,\tstA)$ is called
    \begin{enumerate}
    \item a \emph{Poisson product test configuration} if $(\tstX,\Poiss,\tstA)|_0$ is isomorphic to $(X,\poiss,\alpha)$, ignoring the $\C^*$ action; and is
    \item \emph{trivial} if $(\tstX,\Poiss,\tstA)|_0$ is isomorphic to $(X,\poiss,\alpha)$ with the trivial action of $\C^*$.
    \end{enumerate}
When $\poiss=0$, we say the pair $(\tstX,\tstA)$ is a \emph{product} or \emph{trivial} K\"ahler test configuration, respectively.  Dropping the K\"ahler classes, we obtain the notion of a \emph{product} or \emph{trivial} pre-test configuration.
\end{defn}

Evidently we have the implications
\begin{gather*}
(\tstX,\Poiss,\tstA) \textrm{ is a trivial Poisson  test configuration} 
\\
\Downarrow
\\
(\tstX,\Poiss,\tstA) \textrm{ is a Poisson product test configuration }
\\
\Downarrow
\\
(\tstX,\tstA) \textrm{ is a product test configuration}
\\
\Downarrow
\\
\tstX \textrm{ is a product pre-test configuration}.
\end{gather*}
but the converses fail, in general: not every  product pre-test configuration admits a compatible K\"ahler class or Poisson structure, let alone one that makes it into a Poisson product test configuration.  These subtleties can be profitably understood in terms of one-parameter groups of automorphisms, as we now explain; see Proposition \ref{prop:poisson-1ps} below for a summary.



\subsection{Product test configurations vs.\ one-parameter subgroups} 
In this subsection, we recall the standard construction of pre-test configurations from  one-parameter groups of automorphisms, and then explain how the construction interacts with K\"ahler classes and Poisson structures, leading to a characterization of Poisson product test configurations.

\subsubsection{Pre-test configurations from one-parameter subgroups}
Let us denote by
\begin{align*}
U_0 &:= \PP^1\setminus \{\infty\} \cong \C & U_\infty &:= \PP^1 \setminus \{0\} \cong \C
\end{align*}
the standard charts on $\PP^1$.   Given a one-parameter subgroup 
\[
\rho : \C^* \to \Aut(X),
\]
equip the product $X \times U_0$ with the diagonal $\C^*$-action induced by $\rho$ and the standard action on $U_0$, and glue it to the trivial family $X \times U_\infty$ over $\C^* = U_0 \cap U_1$, using $\rho$ as the clutching function.  In this way we obtain a pre-test configuration for $X$, which we denote by
\[
X \rtimes_\rho \PP^1 \to \PP^1
\]
together with canonical equivariant isomorphisms
\begin{align*}
(X\rtimes_\rho\PP^1)|_0 &\cong (X,\rho)   & (X\rtimes_\rho\PP^1)|_\infty \cong (X,\mathrm{id}),
\end{align*}
i.e.~the $\C^*$ action on $X\rtimes_\rho \PP^1$ interpolates between the trivial action over infinity, and the action by $\rho$ over zero.  Thus, $X\rtimes_\rho \PP^1$ is a product pre-test configuration.  In fact, all product pre-test configurations arise in this way:

\begin{lemma}\label{prop:subgroups}
The following statements hold.
\begin{enumerate}
\item Let $\tstX \to \PP^1$ be a product pre-test configuration for $X$.  Then there exists a one-parameter subgroup $\rho : \C^* \to \Aut(X)$ and an isomorphism of pre-test configurations
\[
\tstX \cong X\rtimes_\rho \PP^1
\]
\item If $\rho,\tilde \rho :\C^* \to \Aut(X)$ are any one-parameter subgroups, then  restriction to the fibre at $\infty$ gives a bijection between isomorphisms of pre-test configurations $X\rtimes_\rho \PP^1 \to X\rtimes_{\tilde \rho}\PP^1$ and elements $\phi \in \Aut(X)$ such that the limit
    \[
    \lim_{t \to 0} \tilde \rho(t) \phi \rho(t)^{-1} \in \Aut(X)
    \]
    exists.   In particular, $X \rtimes_\rho \PP^1$ and $X \rtimes_{\tilde \rho} \PP^1$ are isomorphic if and only if  $\rho$ and $\tilde \rho$ are conjugate by an element of $\Aut(X)$.
\end{enumerate}
\begin{proof}
For the first statement, the ``if'' part is trivial.  For the converse, suppose that $\tstX\to \PP^1$ is a pre-test configuration such that $\tstX|_0$ is isomorphic to $X$.  Then $\tstX \to \PP^1$ is a proper holomorphic submersion with all fibres isomorphic; hence it is locally trivial by a theorem of Fischer--Grauert~\cite{FG}.  Combining a local trivialization near $0$ with the $\C^*$ action over $U_0\setminus\{0\}$ we deduce that $\tstX|_{U_0}$ is isomorphic to the trivial bundle $X\times U_0$ equipped with \emph{some} $\C^*$-action $\psi : \C^* \to \Aut(X\times U_0)$ covering the standard action on $U_0$.  Since $0 \in U_0$ is a fixed point, restriction to the fibre over zero gives a one-parameter subgroup $\rho:= \psi|_0 : \C^* \to \Aut(X)$.  We claim that there exists an automorphism of the family, i.e.~a gauge transformation $\phi : U_0 \to \Aut(X)$, conjugating $\psi$ to the diagonal action defined by $\rho$.  
 
 To this end, consider the trivial bundle $\Aut(X) \times U_0 \to \C$ equipped with the $\C^*$-action given by $\lambda \cdot (\gamma,z) = (\psi(\lambda,z) \gamma \rho(\lambda)^{-1},\lambda z)$.  Then a section $\phi : U_0 \to \Aut(X)$ is $\C^*$-equivariant if and only if it defines a gauge transformation conjugating $\psi$ to $\rho$.  Note that since $\psi|_0=\rho$, the point $(\mathrm{id}_X,0) \in \Aut(X)\times U_0$ is fixed by the action.  Since the fibre $\Aut(X)\times\{0\}$ is preserved, we have a $\C^*$-equivariant splitting
 \[
 T_{(\mathrm{id}_X,0)}(\Aut(X)\times U_0) \cong T_{\mathrm{id}_X}\Aut(X) \oplus T_0U_0
 \]
 of the tangent space.  Under a linearization of the $\C^*$-action, the inclusion of $T_0U_0$ in the direct sum corresponds to a germ of a $\C^*$-equivariant section, and then this germ extends uniquely to all of $\C$ by equivariance, giving the desired gauge transformation. 

For the second statement, note that by $\C^*$-equivariance and continuity, a morphism $X \rtimes_\rho \PP^1 \to X \rtimes_{\tilde \rho}\PP^1$ is determined by its value on the fibre over $\infty$, which is an automorphism $\phi \in \Aut(X)$.  The latter extends to a constant gauge transformation over the equivariant trivialization $X \times U_\infty$.  Switching to the trivialization $X \times U_0$ using the clutching functions, the gauge transformation becomes $\tilde \rho(t) \phi \rho(t)^{-1}$ for $t\neq0 \in U_0$.  Thus $\phi$ determines an isomorphism if and only if this gauge transformation extends over $0 \in U_0$, which is evidently equivalent to the existence of the limit in the statement of the lemma.  The limiting automorphism then conjugates $\rho$ to $\tilde \rho$, giving  the remaining statement.
\end{proof}

\end{lemma}

\subsubsection{K\"ahler classes and one-parameter subgroups}
To incorporate a K\"ahler structure into the pre-test configuration defined by a one-parameter subgroup, we must restrict our attention to \emph{reduced} one-parameter subgroups:

\begin{lemma}
    Let $\alpha$ be a K\"ahler class on $X$, and let $\rho : \C^* \to \Aut(X)$ be a one-parameter subgroup.  Then the pre-test configuration $X \rtimes_\rho \PP^1 \to \PP^1$ admits a K\"ahler class $\tstA$ making the pair $(X\rtimes_\rho \PP^1,\tstA)$ into a test configuration for $(X,\alpha)$ if and only if $\rho$ takes values in the subgroup $\Aut_{\rm red}(X) < \Aut(X)$ of reduced automorphisms.
\end{lemma}
\begin{proof}
To see the necessity, suppose that $\tstX=X\rtimes_\rho \PP^1$ is K\"ahler. Then from the degeneration of the Leray spectral sequence for the map $X\rtimes_\rho \PP^1 \to \PP^1$, we deduce that the restriction $H^1(\tstX;\Z) \to H^1(\tstX|_t;\Z)$ is an isomorphism of pure Hodge structures for all $t \in \PP^1$, and hence induces a $\C^*$-equivariant isomorphism of Albanese tori $\Alb(\tstX) \to \Alb(\tstX|_t) \cong \Alb(X)$ for $t = 0,\infty$.  Since $\C^*$ acts trivially on the fibre over infinity, while it acts on the fibre over zero through $\rho$, we deduce that the action on $\Alb(X)$ induced by $\rho$ is trivial, i.e.~$\rho$ takes values in $\Aut_{\rm red}(X)$.

For the converse, note that as a complex manifold,   $\tstX \cong  X\times_{\C^*} \cO_{\PP^1}(1)^{\times}$  is an $X$-fibre bundle associated to the principal $\C^*$-bundle $\cO_{\PP^1}(1)^{\times}\to \PP^1$,   with action induced by the natural $\C^*$-action on $\cO_{\PP^1}(1)\to\PP^1$.
In order to turn $\tstX$ into a K\"ahler test configuration, we need to construct  a compatible K\"ahler metric on $\tstX$. To this end, we use that $\rho$ is a $\C^*$ action in the reduced automorphism group of $X$: in this case, for any $\Sph^1$-invariant K\"ahler metric $\omega \in \alpha$  one can build (e.g.\ via the momentum construction in \cite{AJL}) an $\Sph^1$-invariant K\"ahler metric $\Omega$ on $\tstX$ satisfying $\Omega_{|_{X \times \{\tau\}}}= \omega$, as desired. 
\end{proof}

\subsubsection{Poisson structures and one-parameter subgroups} %
Now suppose that $\poiss$ is a Poisson structure on $X$.  We wish to endow a pre-test configuration of the form $X\rtimes_\rho \PP^1$ with a Poisson structure making it into a Poisson pre-test configuration.  To this end, note first that in the trivialization $X \times U_\infty$ over $U_\infty = \PP^1\setminus\{0\}$, the $\C^*$-invariance and continuity require that the Poisson structure on each fibre be identified with $\poiss$, i.e.~in this trivialization the family of Poisson structures is constant.  It then remains to extend the Poisson structure  over $0 \in \PP^1$.  To assess whether such an extension exists, we switch to the  trivialization $X \times U_0$ using the clutching function $\rho$.  Then the Poisson structure on the fibre $X \cong X \times \{t\} \subset X \times U_0$ for $t \neq 0$ is identified with $\rho(t) \cdot \poiss$. Hence the existence of the extension over $0 \in U_0$ is equivalent to the following condition on $\rho$:
\begin{defn}\label{def:poiss-compatible-1ps}
    A one-parameter subgroup $\rho : \C^* \to \Aut(X)$ is \emph{$\poiss$-compatible} if the following limit exists:
    \[
    \lim_{t \to 0} \rho(t)\cdot\poiss \in H^0(X,\wedge^2T^{1,0}X)
    \]
\end{defn}
Even if the limiting Poisson structure on the fibre over 0 exists, it need not be isomorphic to $\poiss$, let alone by an isomorphism respecting the K\"ahler class, i.e.~the resulting Poisson pre-test configuration need not be a Poisson product test configuration.   The existence of such a Poisson K\"ahler isomorphism is evidently equivalent to the following stronger condition on the limit:
\begin{defn}
    A one-parameter subgroup $\rho :\C^* \to \Aut(X)$ is \emph{$\poiss$-stabilizing relative to $\alpha$} if 
    \[
    \lim_{t \to 0} \rho(t)\cdot\poiss \in \Aut(X,\alpha)\cdot \poiss,
    \]
    or equivalently, $\rho$ is conjugate, by an element of $\Aut(X,\alpha)$, to a one-parameter subgroup of $\Aut(X,\poiss,\alpha)$.
\end{defn}
Hence, with these definitions in place, the discussion above gives the following.
\begin{prop}\label{prop:poisson-1ps}
Let $(\tstX,\Poiss,\tstA)$ be a Poisson test configuration such that $(\tstX,\tstA)$ is a product test configuration.  Then the following statements hold:
\begin{enumerate}
    \item $(\tstX,\Poiss,\tstA)$ is isomorphic to the Poisson K\"ahler test configuration defined by a $\sigma$-compatible one-parameter subgroup $\rho : \C^* \to \Aut_{\rm red}(X)$.
    \item $(\tstX,\Poiss,\tstA)$ is a Poisson product test configuration if and only if, in addition, the one-parameter subgroup $\rho$ stabilizes $\poiss$ relative to $\alpha$.  In this case, we have a (possibly non-equivariant) isomorphism of families
    \[
    (\tstX,\Poiss)|_{U_0} \cong (X,\poiss)\times U_0.
    \]
\end{enumerate}
\end{prop}

Using this result, we can easily construct examples of Poisson test configurations that are not Poisson product configurations, even though their underlying test configuration is a product:
\begin{ex}
    Let $Y$ be a compact K\"ahler manifold and let $X = \PP(T^*_{1,0}Y\oplus \cO_Y)$ be the compact K\"ahler manifold obtained by compactifying the holomorphic cotangent bundle of $Y$ by adding a projective  hyperplane at infinity.  Then the holomorphic symplectic structure on $T^*Y$ extends to a Poisson structure $\poiss$ on $X$.  The action of $\C^*$ on $T^*Y$ by dilation in the fibres extends to $X$, giving a one-parameter subgroup $\rho : \C^* \to \Aut_{\rm red}(X)$, and $\poiss$ is homogeneous of weight one with respect to this action.  Hence $\lim_{\lambda \to 0} \poiss = 0 \notin \Aut(X)\cdot \poiss$, so $\rho$ is $\poiss$-compatible but not $\poiss$-stabilizing.  
\end{ex}

\subsection{Poisson K-stability}
\label{sec:stability}
The notion of K-stability is phrased in terms of the following numerical invariant of test configurations.
\begin{defn}\label{d:DF} The \emph{Donaldson-Futaki invariant} of a K\"ahler test configuration $(\tstX, \tstA)$ of $(X, \alpha)$ is the intersection number
\begin{align}
 {\rm DF}(\tstX, \tstA) := - \Big(\frac{\left[c_1(\tstX)-\pi^*(c_1(\PP^1))\right] \cdot \tstA ^n}{n!}\Big)_{\tstX}
 + \Big(\frac{c_1(X) \cdot \alpha^{n-1}}{(n-1)!}\Big)_{X} \Big(\frac{{\rm Vol}(\tstX, \tstA)}{{\rm Vol}(X, \alpha)}\Big).   \label{eq:DF}
 \end{align}
 \end{defn}
In the above definition, the intersection of de Rham classes can be performed on any smooth Hironaka resolution $R: \widetilde{\tstX} \to \tstX$.  
It follows easily from the above that for a trivial test configuration $(\tstX_0 = X\times \PP^1, \tstA_{\varepsilon}= \alpha + \varepsilon c_1(\cO_{\PP^1}(1)))$ we have ${\rm DF}(\tstX_0, \tstA_{\varepsilon})=0$.

\begin{defn}
A compact Poisson K\"ahler manifold $(X,\poiss,\alpha)$ is called \emph{K-semistable} if
    \[
    \DF(\tstX,\tstA) \ge0
    \]    
    for  every Poisson K\"ahler test configuration $(\tstX,\Poiss,\tstA)$.  It is called \emph{K-polystable} if, in addition, equality holds if and only if $(\tstX,\Poiss,\tstA)$ is a Poisson product test configuration.
\end{defn}

Using Proposition \ref{prop:poisson-1ps}, we may relate the notion of Poisson K-polystability to the more classical notion of stability in geometric invariant theory, which we now recall.
\begin{defn}\label{d:GIT} Let $\G$ be a reductive complex Lie group acting linearly on a complex vector space $W$. A non-zero point $w\in W$ is called 
\begin{enumerate}
\item  \emph{polystable} if the $\G$-orbit  $\G \cdot w$ of $w$ is closed in $W$; 
\item  \emph{stable} if it is polystable and the stabilizer $\G_w$ is  finite.  
\item  \emph{semistable} if the closure $\overline{\G\cdot w}$ does not contain $0$;
\item \emph{unstable} if $0 \in \overline{\G\cdot w}$.
\end{enumerate}
\end{defn}
We further recall that, according to the Hilbert--Mumford criterion, $w$ is polystable if and only if for every one-parameter subgroup $\rho : \C^* \to \G$ such that the limit $\lim_{\lambda \to 0} \rho(\lambda)\cdot w = w_0$ exists, we have $w_0 \in \G\cdot w$.  We then have the following result:

\begin{prop}\label{p:Kahler/Poisson-K-polystable} 
Let $(X,\alpha)$ be a K-polystable K\"ahler manifold such that $\Aut_{\rm red}(X)$ is reductive, and let $\poiss$ be a holomorphic Poisson structure on $X$.  Then the triple $(X,\poiss,\alpha)$ is Poisson K-polystable if and only if $\poiss$ is a polystable point for the action of $\Aut_{\rm red}(X)$ on $H^0(X, \Wedge^2 T^{1,0}_X)$.
\end{prop}

\begin{proof}
Let $(\tstX,\Poiss,\tstA)$ be a Poisson test configuration.  Since $(X,\alpha)$ is K-polystable, we have $\DF(\tstX,\tstA)=0$ if and only if $(\tstX,\tstA)$ is a product test configuration. Moreover, by Lemma \ref{lem:fujiki}, the $\Aut(X,\alpha)$-orbit of $\poiss$ is a finite union of $\Aut_{\rm red}(X)$-orbits. 
 The result therefore follows by combining Proposition \ref{prop:poisson-1ps} with the  Hilbert--Mumford criterion for the action of the reductive group $\Aut_{\rm red}(X)$ on $H^0(X, \wedge^2 TX)$.
\end{proof}

The proposition allows us to easily check that various well known Poisson varieties are K-stable:

\begin{ex} Since $\wedge^2 T^{1,0} \PP^2 \cong \cO_{\PP^2}(3)$, a nonzero Poisson structure on $\PP^2$ vanishes on a cubic curve, which determines it up to rescaling by a constant.  It follows from  Proposition~\ref{p:Kahler/Poisson-K-polystable}  and the analysis of polystable bivectors $0\neq \poiss \in H^0(\PP^2, \cO_{\PP^2}(3))$ under the action of ${\rm PGL}(3, \C)$ (see e.g. \cite{nakamura}, Table 1) 
that $(\PP^2, \cO_{\PP^2}(1), \poiss)$ is Poisson K-polystable if and only if the zero locus $\poiss=0$ corresponds to one of the following two cases:
\begin{enumerate}
\item a triangle of lines: these are polystable toric Poisson structures, i.e. $\Aut_{\circ}(\PP^2, \sigma)\cong\C^* \times \C^*$.
\item smooth elliptic curve:  these are stable Poisson structures, i.e.  $\Aut_{\circ}(\PP^2, \sigma)=\{1\}$.
\end{enumerate}
\end{ex}

\begin{ex}\label{ex:P1xP1}
If $X=\PP^1\times \PP^1$, a Poisson structure is a section of $K^{-1}_{\PP^1\times \PP^1}\cong \cO_{\PP^1\times\PP^1}(2,2)$.  Using the polarization defined by $\cO_{\PP^1\times \PP^1}(p,q), \, p,q>0$, which is K-polystable by \cite{BDL}, the Poisson K-polystability then corresponds to the polystability of $(2,2)$-divisors under the action of the group $G = \mathrm{SL}(2,\mathbb{C})\times \mathrm{SL}(2,\mathbb{C})$, which can be determined using the methods of \cite{KOP-GIT}.  We find that a Poisson structure on $\PP^1\times \PP^1$ is K-polystable if and only if its vanishing locus lies in the $G$-orbit of one of the following:
\begin{enumerate}
    \item a smooth curve of bidegree $(2,2)$, i.e.~an elliptic curve; these are stable Poisson structures, with $\Aut_\circ(\PP^1\times\PP^1,\sigma) = \{1\}$.
    \item the toric boundary divisor $(\{0,\infty\}\times \PP^1) \cup (\PP^1\times\{0,\infty\})$; in this case, the Poisson structure is polystable with automorphism group $\Aut_\circ(\PP^1\times\PP^1,\sigma) \cong \mathbb{C}^*\times\mathbb{C}^*$.
    \item the union of the diagonal and the graph of an automorphism $1\neq\phi \in \mathbb{C}^*<\Aut(\PP^1)$ that fixes $\{0,\infty\}$ pointwise; these Poisson structures are polystable with automorphism group $\Aut_\circ(\PP^1\times\PP^1,\sigma) \cong \mathbb{C}^*$
    \item the double diagonal; this case is polystable with automorphism group $\Aut_\circ(\PP^1\times\PP^1,\sigma)\cong\mathrm{PGL}(2,\mathbb{C})$.
\end{enumerate}
\end{ex}

\begin{ex}\label{ex:sklyanin} %
    Given a  pencil of quadrics in $\PP^3$, there is unique (up to scale) nonzero Poisson structure  on $\PP^3$ for which every quadric in the pencil is a Poisson subvariety.  According to~\cite{Avritzer1999,Wall1983}, a pencil of quadricks on projective space is GIT-stable if and only if its discriminant is reduced, which for $\PP^3$ is equivalent to the statement that exactly four members of the pencil have an isolated singularity, and the rest are smooth.  The corresponding Poisson structures are then also K-polystable by Proposition~\ref{p:Kahler/Poisson-K-polystable}.  The vanishing locus of such a Poisson structure is the disjoint union of an elliptic normal curve (the base locus of the pencil), and the four isolated singularities of the quadrics; these are the celebrated elliptic Poisson structures introduced by Sklyanin~\cite{Sklyanin1982}.  Higher-dimensional  Poisson structures associated with elliptic curves were introduced by Feigin--Odesskii~\cite{Feigin1989,Feigin1998} and we expect that they are all stable.              \end{ex}

\subsection{Poisson K-stability and GK geometry}
Suppose that $(\tstX,\Poiss,\tstA)$ is a  Poisson K\"ahler test configuration of $(X,\poiss,\alpha)$ for which the total space $\tstX$ is smooth, and suppose that $\tstF \in \GK_{\Poiss,\tstA}$ is a symplectic GK structure that is invariant under the action of $S^1$ on $\tstX$.  Then the $\tstF$-conjugate complex structure gives a second complex manifold $\tstX_-$, and the $S^1$-action induces a holomorphic $\C^*$ action on $\tstX_-$. Since the fibres of $\pi$ are holomorphic Poisson subvarieties of $(\tstX,\Poiss)$, Lemma \ref{lem:GK-restrict} implies that they are also holomorphic in $\tstX_-$, and hence the induced map $\tstX_- \to \PP^1$ is also holomorphic.  Moreover, by construction, the action of $S^1$ (and hence of $\C^*$) on the fibre over $\infty$ is trivial.

As a result, $(\tstX_-,\Poiss_-,\tstA)$ is a candidate for a smooth test configuration for the Poisson K\"ahler manifold $(X_-,\poiss_-,\alpha) := (\tstX_-,\Poiss_-,\tstA)|_\infty$; note that this is the holomorphic Poisson manifold conjugate to $(X,\poiss,\alpha)$ by the symplectic GK structure $F:=\tstF|_\infty$.  However, there is no reason \emph{a priori} why $\tstX_-$ should be K\"ahler, and even if so, there is no guarantee that $\alpha$ will be a K\"ahler class.  Fortunately, the following gives a simple condition to guarantee that $(\tstX_-,\Poiss_-,\tstA)$ is indeed a Poisson K\"ahler test configuration.

\begin{lemma}\label{l:GK-test}
    Suppose that $\tstX_-$ admits a K\"ahler metric, and that $\alpha \in H^{1,1}(X_-, \C)$.  Then the following statements hold:
    \begin{enumerate}
        \item The triple $(\tstX_-,\Poiss_-,\tstA)$ is a Poisson K\"ahler test configuration for $(X_-,\poiss_-,\alpha)$.
        \item We have the equality 
    \[
        \DF(\tstX_-,\tstA) = \DF(\tstX,\tstA)
    \]
    of the Donaldson--Futaki invariants.
    \item $(\tstX_-,\Poiss_-,\tstA)$ is a product (respectively, trivial) Poisson test configuration if and only if $(\tstX,\Poiss,\tstA)$ is a product (resp.~trivial) Poisson test configuration.
    \end{enumerate}
\end{lemma}

\begin{proof}
As explained above,  Lemma~\ref{lem:GK-restrict} implies that the morphism $\pi : \tstX_- \to \PP^1$ is a holomorphic submersion away from $0$. Hence it is topologically locally trivial.  In particular, if $D := \pi^{-1}(0) \subset \tstX_-$ denotes the central fibre, then the restriction $H^2(\tstX_-\setminus D, \C) \to H^2(X_-, \C)$ is an isomorphism, and from the Gysin/residue spectral sequence we deduce that the kernel of the map $H^2(\tstX_-, \C) \to H^2(X_-, \C)$ is generated by fundamental classes of irreducible components of the central fibre $\tstX_-|_0$.  It follows that $\tstA \in H^{1,1}(\tstX_-, \R)$ if and only if $\alpha \in H^{1,1}(X_-, \R)$.  

To see that $\tstA$ is a K\"ahler class, note that $\tstA$ is represented by the GK symplectic form $\tstF$, which in particular tames $\tstX_-$. Hence if $\omega$ is any K\"ahler form on $\tstX_-$, the form $\tstF+t\omega$ tames $\tstX_-$ for all $t \ge 0$.  It follows that $\int_{Y}(\tstF+t\omega)^{\dim Y} > 0$ for all complex analytic subvarieties $Y\subset \tstX$, and hence $\tstA$ is a K\"ahler class by the  Demailly--Paun criterion \cite[Theorem 4.2(ii)]{DP}.

For the statement about products/triviality (and the triviality of $\tstX_- \to \PP^1$ at infinity),  it suffices to argue that $\tstX\to \PP^1$ is locally trivial over an open subset $U \subset \PP^1$ if and only if $\tstX_- \to \PP^1$ is so.  The statement is symmetric under interchanging the roles of $X$ and $X_-$, so it suffices to prove the ``only if'' part.  Thus, assume given a trivialization $(\tstX,\Poiss)|_U \cong (X,\poiss) \times U$.  We may then view the GK structure on the fibres as a family of two-forms $F_z := \tstF|_z$ on $X$, all lying in the same generalized K\"ahler class.  Since in a symplectic GK structure, the isomorphism class of both holomorphic Poisson manifolds depends only on the generalized K\"ahler class (cf. Remark \ref{rem:Poisspreserved}) we deduce that the fibres of $(\tstX_-,\Poiss_-)|_U$ are also pairwise isomorphic, and the result follows.

Since the complex manifolds $\tstX$ and $\tstX_-$ are tamed by the same symplectic form $\tstF$,  we have $c_1(\tstX)=c_1(\tstX_-)$.  Hence the equality of the Donaldson--Futaki invariants is immediate from the definition \eqref{eq:DF}.
\end{proof}

Using this result, we obtain a sort of invariance of K-polystability under Hamiltonian deformation, at least when we probe with smooth test configurations.
\begin{lemma}\label{lem:gualtieri-def-stable}
    Let $(X,\poiss,\alpha)$ be a holomorphic Poisson K\"ahler manifold with $H^{2,0}(X, \C)=0$, and let $(X_-(t),\poiss_-(t),\alpha)$ for $|t|<\epsilon$ be its Hamiltonian deformation in the sense of Definition \ref{def:ham-def}.  If the triple $(X_-(t),\poiss_-(t),\alpha)$ is K-polystable on smooth test configurations for all $0 < |t| \ll \epsilon$, then so is $(X,\poiss,\alpha)$.
\end{lemma}

\begin{proof}
    By averaging, any smooth test configuration $(\tstX,\Poiss,\tstA)$ admits an $S^1$-invariant K\"ahler metric.  Moreover, restriction gives an isomorphism $H^{0,2}(\tstX, \C)\to H^{0,2}(X_-, \C)=0$. Hence we may apply the Hamiltonian deformation construction of   \autoref{thm:gg-def} to obtain a family of $S^1$-invariant $GK$ structures $\tstF(t)$ which conjugate $(\tstX,\Poiss,\tstA)$ to a family of triples $(\tstX(t),\Poiss(t),\tstA)$.  Since the class of K\"ahler manifolds (and their Hodge numbers) are invariant under small deformation, Lemma \ref{l:GK-test} applies and we deduce that $(\tstX(t),\Poiss(t),\tstA)$ is a Poisson K\"ahler test configuration for the family $(X(t),\poiss(t),\alpha)$ obtained from the induced GK structures $F(t)=\tstF(t)|_\infty \in \GK_{t\poiss,\alpha}$, with the same Donaldson--Futaki invariant, which is a product if and only if $(\tstX,t\Poiss,\tstA)$ is.  The result follows since K-polystability is invariant under rescaling the Poisson structure.
\end{proof}

With this result in hand, we can produce an example of a Poisson variety $(X,\poiss)$ that satisfies the Poisson K-polystablility condition for smooth test configurations, even though the underlying variety $X$ is not K-polystable, and the Poisson structure $\poiss$ is unstable for the action of $\Aut(X)$.  Note that this does not contradict Proposition \ref{p:Kahler/Poisson-K-polystable}. 
\begin{ex}\label{p:F_2-poisson-stability}
    Let $X = \PP(T^*\PP^1\oplus \cO_{\PP^1})$ be the second Hirzebruch surface, and let $\poiss$ be the Poisson structure on $X$ extending the holomorphic symplectic structure on $T^*\PP^1$.   As explained by Hitchin~\cite[p. 7]{Hitchin}, the Hamiltonian deformation has the form $X(t)\cong \PP^1 \times \PP^1$ for $t \neq 0$.  Since $\PP^1\times \PP^1$ is cscK for any polarization, we deduce $(X(t),\alpha)$ is K-polystable for $t \neq 0$; see \cite{BDL, stoppa}.  Moreover, since the Poisson structures on $X(t)$ have the same vanishing set for all $t$, we deduce that for $t \neq 0$, the vanishing locus of $\poiss(t)$ is identified with a non-reduced curve of bidegree $(2,2)$ on $\PP^1\times \PP^1$.  Hence $(X(t),\poiss(t),\alpha)$  is K-polystable for $t \neq 0$ by Example \ref{ex:P1xP1}.  It therefore follows from Corollary 
    \ref{lem:gualtieri-def-stable} that $(X,\poiss,\alpha)$ is K-polystable on smooth test configurations.

The underlying surface $X$ itself is K-unstable with respect to any polarization: indeed, the one-parameter subgroup $\rho \to \Aut(X)$ acting fiberwisely  gives rise (when suitably oriented) to a product test configuration $X \rtimes_\rho \PP^1$ with negative Donaldson--Futaki invariant; this also follows for instance from Calabi's construction~\cite{calabi} of extremal K\"ahler metrics with non-constant scalar curvature. However,  this test configuration is  not Poisson  since $\lim_{t\to 0}\rho(t)\cdot \poiss = \infty$. Meanwhile, for the inverse one-parameter subgroup $\rho^{-1}$, we have $\lim_{t \to 0}\rho^{-1}(t)\cdot \poiss = 0$, so that $X\rtimes_{\rho^{-1}}\PP^1$ is a non-product Poisson test configuration whose underlying test configuration is a product, but with strictly positive Donaldson--Futaki invariant.  Hence, while $\rho$ and $\rho^{-1}$ destabilize $X$ and $\poiss$ separately, neither of them destabilizes the pair $(X,\poiss)$.
\end{ex}

\begin{rem} \label{rem:singular-GK}
The results in this subsection apply to smooth test configurations, but can also be used to obtain some information for singular test configurations.  Namely, suppose $(\tstX,\Poiss,\tstA)$ is a singular Poisson test configuration that admits a Poisson resolution of singularities $\phi : (\widetilde \tstX,\widetilde \Poiss) \to (\tstX,\Poiss)$ by a sequence of blowups, and a family of K\"ahler classes $\widetilde\tstA_\epsilon := \phi^*\tstA - \epsilon E$ for $\epsilon > 0$, where $E\subset \widetilde \tstX$ is the exceptional divisor.  This gives a family of test configurations $(\widetilde \tstX,\widetilde\tstA_\epsilon)$ such that $\lim_{\epsilon\to 0}\DF(\widetilde \tstX,\widetilde\tstA_\epsilon) = \DF(\tstX,\tstA)$. 

However, there are singular Poisson varieties that do not admit Poisson resolutions, i.e.~the Poisson analogue of Hironaka's theorem fails.  Hence, to obtain a full correspondence between the stability properties of the two holomorphic Poisson structures underlying a GK structure, it seems necessary to extend the constructions in this subsection to singular (but normal) test configurations.  We believe that such an extension is possible,  but since a theory of singular generalized K\"ahler spaces has yet to be developed, we leave this extension for future work.

\end{rem}


%
%
%
%
\section{Conjectural existence} 

To give a geometric motivation for Conjecture~\ref{c:intro} from the introduction, we  check that, at least for a special class of  test configurations, the Poisson K-polystability condition yields information on the slope at infinity of the Mabuchi $1$-form constructed in \cite{ASU}, see Lemma~\ref{l:product} and Remark~\ref{r:slope} below for the  expected relationship between this slope and Poisson $K$-polystability.

We first prove a general lemma on the Futaki character $\Futc_{\poiss, \alpha}: {\rm Lie}(\Aut_{\rm red}(X, \poiss)) \to \C$ associated to $\GK_{\poiss, \alpha}$ (see \cite[Thm.\ 4.18]{ASU}), where $\Aut_{\rm red}(X, \poiss):= \Aut_{\rm red}(X) \cap \Aut(X, \poiss)$ stands for the group of reduced Poisson automorphisms of $(X, \poiss)$ (see \cite[Def.\ 4.12 \& Rem.\ 4.10-(2)]{ASU}).
\begin{lemma}\label{l:GK-Futaki=K-Futaki} Let $(X,\poiss)$ be a compact holomorphic Poisson K\"ahler manifold and let $\alpha$ a K\"ahler deRham class such that there exists a symplectic GK structure $F\in \GK_{\poiss,\alpha}$. Suppose $V$ is a $F$-Hamiltonian vector field which is also real holomorphic, and hence corresponds to an element of the Lie algebra of $\Aut_{\rm red}(X,\poiss)$.
Then $\Futc_{\poiss, \alpha}(JV)=\Futc_{\alpha} (JV)$ where $\Futc_{\alpha}$  is the K\"ahler Futaki character of $(X, \alpha)$.
\end{lemma}
\begin{proof} Notice that $V$ is a Killing symmetry of the biHermitian structure $(g, J, I)$ defined by $(F, J)$. It thus generates a compact torus $\T \subset \Aut_{\rm red}(X, \poiss)$, see \cite[Def.\ 4.12]{ASU}. As $\T$ preserves $F$ and $\T<\Aut_{\rm red}(X, \poiss)$, $\T$  is $F$-Hamiltonian by \cite[Lemma 4.7]{ASU}. We denote by
$\psi_{(V, F)}$ the Hamiltonian potential of $V$ with zero mean with respect to $F$.  Then (see \cite[Thm.4.18]{ASU}) $\Futc_{\poiss, \alpha}(JV)=\int_X \psi_{(V, F)}\Gscal(F, J) dV_F$. As $\alpha$ is a K\"ahler class and $\T < \Aut_{\rm red}(X)$, $X$ admits a $\T$-invariant K\"ahler metric $\omega \in \alpha$ with $\T< \Ham_{\omega}$. We then have $\Futc_{\alpha}(JV)=\int_X \psi_{(V, \omega)} \Scal_{(\omega, J)}dV_{\omega}$, where $\psi_{(V,\omega)}$ is the normalized Hamiltonian  potential of $V$ with respect to $\omega$ and $\Scal_{(\omega, J)}$ is the scalar curvature of $(\omega, J)$. By  the $\T$-equivariant Moser Lemma, we can map $(F, J)$  to a $\T$-invariant GK structure $(\omega, \tilde J)$ where $\tilde J$ is an element of the space $\AGK_{\omega}^{\T}$ of $\T$-invariant almost complex structures tamed by $\omega$. The lemma follows by  \cite[Prop.\ 4.16]{ASU} (and the fact that $\Futc_{\poiss, \alpha}(JV)$ is invariant under diffeomorphisms) . \end{proof}

\begin{lemma}\label{l:product} If there exists $\lambda \in \C$ such that $(X, \lambda\poiss, \alpha)$ admits a cscGK structure then ${\rm DF}(\tstX, \tstA)=0$ for any product $\poiss$-Poisson K\"ahler test configuration $(\tstX, \Poiss, \tstA)$.
\end{lemma}
\begin{proof} Let $(\tstX, \Poiss, \tstA)$ be a product Poisson test configuration of $(X, \poiss, \alpha)$.  In this case,  $\tstX$ is smooth,  $\pi$ is a submersion, and  the central fiber $(X_0, \poiss_0, \alpha_0)$ is Poisson biholomorphic to $(X, \poiss, \alpha)$.  Furthermore, $(X_0, \poiss_0, \alpha_0)$ is invariant under the $\C^*$-action of $\tstX$ and, as observed in the proof of Proposition~\ref{p:F_2-poisson-stability},  the induced holomorphic fundamental vector field $-\xi_0 + \i J\xi_0$ on $X_0$ by the $\C^*$-action $\rho$ belongs to $\Aut_{\rm red}(X_0)$, and hence (see \cite[Rem.\ 4.10(2)]{ASU}) in $\Aut_{\rm red}(X_0, \lambda\poiss_0)$ for any $\lambda\in \C$.
We know that  for a product K\"ahler test configuration ${\rm DF}(\tstX, \tstA) =  -\Futc_{\alpha_0}(\xi_0)$ where $\Futc_{\alpha_0}$ is the (K\"ahler) Futaki invariant on $(X_0, \alpha_0)$. As $(X_0, \lambda\poiss_0, \alpha_0) \cong (X, \lambda\poiss, \alpha)$ admits a cscGK structure by assumption,  and the latter is invariant under a maximal torus in $\Aut_{\rm red}(X_0, \lambda \poiss)$ (see \cite[Thm.\ 4.13]{ASU})  and $\Futc_{\lambda\poiss,\alpha}=0$ (see \cite[Thm.\ 4.18]{ASU}), we conclude by Lemma~\ref{l:GK-Futaki=K-Futaki}  that  $\Futc_{\lambda\poiss_0, \alpha_0}(\xi_0)=\Futc_{\alpha_0}(\xi_0)=0$.
\end{proof}

\begin{prop}\label{p:smooth-ray} Let $(\tstX, \Poiss,\tstA)$ be a smooth Poisson K\"ahler test configuration of $(X, \poiss, \alpha)$,  such that the morphism $\pi : \tstX \to \PP^1$ is a smooth submersion. Suppose, furthermore, that $\tstX$ admits an $S^1$-invariant symplectic GK structure $\Omega \in \tstA$ with holomorphic Poisson tensor $\Poiss$, where the (holomorphic Poisson) $S^1$-action on $\tstX$ is induced by the $\C^*$-action $\rho$. Letting $z:=e^{-t + is} \in \C^*$, denote by $\Omega_t:= \rho(z)^*\Omega$ the family of symplectic GK structures on $(\tstX, \Poiss,\tstA)$ and by $F_t := (\Omega_t)_{|_{X_1}} \in \GK_{\poiss,\alpha}$ the induced family of symplectic GK structures on  $X_1:=\pi^{-1}(1) \cong (X, \alpha, \poiss)$.  Then the Mabuchi $1$-form $\boldsymbol{\tau}$ on the space $\GK_{\poiss,\alpha}$ (introduced in \cite[Prop.\ 4.5]{ASU}) satisfies
\[ \lim_{t\to \infty} \boldsymbol{\tau}_{F_t}(\dot{F_t}) =  {\rm DF}(\tstX, \tstA). \]
\end{prop}
\begin{proof} This is a straightforward generalization of the original argument in the K\"ahler case by Ding--Tian~\cite{DT}. Denote by $\hat \xi:= \rho_*\left(\frac{\partial}{\partial t}\right)$ the real Poisson holomorphic vector field on $(\tstX, \Poiss)$ which generates the $\R_+$-action via $\rho$. As $\tstX$ is a K\"ahler manifold and 
$\hat \xi$ vanishes on the fiber $X_{\infty} := \pi^{-1}(\infty)$ (here we have identified $\PP^1 = \{0\}\cup \C^* \cup\{\infty\}$), by \cite[Lemma\ 4.7 \& Rem.\ 4.10-(2)]{ASU}, there are smooth functions $h, f$ on $\tstX$ such that
$\hat \xi = \Omega^{-1}(df) + J_{\tstX} \Omega^{-1} (dh),$ where $J_{\tstX}$ stands for the complex structure of $\tstX$ and $I_{\tstX}$ will later denote the $\Omega$-conjugate complex structure.  By the proof of \cite[Lemma\ 4.8]{ASU},
\begin{equation}\label{infinitesimal}
\mathcal{L}_{\hat \xi} \Omega = dI_{\tstX} d h. \end{equation} 
As the functions $h$ satisfying (\ref{infinitesimal}) are unique up to an additive constant, and using that $\Omega$  and $I_{\tstX}$ are $S^1$-invariant, we can assume that $h$ is a smooth $S^1$-invariant function. Furthermore, 
as $J_{\tstX} \hat\xi = \rho_*\left(\frac{\partial}{\partial s}\right)$ is the generator of $S^1$-action and $\Omega$ is $S^1$-invariant, if follows that $f$ is constant, i.e. 
\begin{equation}\label{gradient} \hat \xi = J_{\tstX}\Omega^{-1}(dh).\end{equation}
It also follows from \eqref{infinitesimal} that 
\begin{equation}\label{dot-Omega} \dot{\Omega}_t:= \frac{d}{dt} \Omega_t = d I_{\tstX} d h_t, \qquad h_t:= \rho(e^{-t+is})^*h.\end{equation}

Recall that the tangent space of ${\GK}_{\poiss, \alpha}$ at $F$ is identified in 
\cite[Rem.\ 2.18]{ASU} with $C^{\infty}(X, \R)/{\R}$ which in  turn is isomorphic to the space $C^{\infty}_{0,F}(X, \R)$ of smooth functions of zero mean with respect to $dV_F :=F^{[n]}= \frac{F^n}{n!}$. The definition of $\boldsymbol{\tau}$ in \cite{ASU} is then
\[ \boldsymbol{\tau}_F(\phi) := -\int_X \phi \Gscal(F, J_0) dV_F, \qquad \phi \in C^{\infty}_{0,F}(X, \R). \]
As observed in \cite[Rem.\ 3.9]{ASU}, the mean value $\overline{\mu}$ of $\Gscal(F,J_0)$  with respect to $dV_F$ is a topological invariant independent of the choice of $F\in \GK_{\poiss, \alpha}$, so we can equivalently define
\[ \boldsymbol{\tau}_F([\psi])= -\int_X \psi\left(\Gscal(F, J_0) -\overline{\mu}\right) dV_F, \qquad [\psi] \in C^{\infty}(X, \R)/\R.\]
With this at hand, in what follows $F_t:=(\Omega_t)_{|_{X_1}}$ will denote the restriction of $\Omega_t$  to the fiber $X_1=\pi^{-1}(1)$ which  we will identify $X_1 \equiv (X, \poiss, \alpha)$ through the properties of the test configuration.  By Lemma~\ref{l:GK-test}, $F_t \in \GK_{\poiss,\alpha}$.
By \eqref{dot-Omega}, $ \dot{F}_t = d I_X d \left((h_t)_{|_{X_1}}\right)$ so that $\dot F_t =[(h_t)_{|_{X_1}}] \in C^{\infty}(X_1, \R)/\R$. 
We thus compute
\[
\begin{split}
\boldsymbol{\tau}_{F_t}(\dot{F}_t) &= -\int_{X_1}h_t\left(\Gscal(F_t, J_0)-\overline{\mu}\right) F_t^{[n]}\\
&= -\int_{X_1}h_t\left(\Gscal(F_t, J_0)-\overline{\mu}\right) \Omega_t^{[n]} \\
&= -\int_{X_1} \rho(z)^*\left(h_{|_{X_{z}}} \left(\Gscal(\Omega_{|_{X_{z}}}, J_{|_{X_z}})-\overline{\mu}\right)\Omega^{[n]}_{|_{X_{z}}}\right) \\
&= -\int_{X_{z}}h_{|_{X_{z}}} \left(\Gscal(\Omega_{|_{X_{z}}}, J_{|_{X_z}})-\overline{\mu}\right)\Omega^{[n]}_{|_{X_{z}}}.
\end{split}
\]
It  follows from the above
\begin{equation}\label{limit} 
\begin{split}
\lim_{t\to \infty} \boldsymbol{\tau}_{F_t}(\dot{F}_t) &= \lim_{z \to 0}-\int_{X_{z}}h_{|_{X_{z}}} \left(\Gscal(\Omega_{|_{X_{z}}}, J_{|_{X_z}})-\overline{\mu}\right)\Omega^{[n]}_{|_{X_{z}}}\\
&= -\int_{X_0} h_{|_{X_{0}}} \left(\Gscal(\Omega_{|_{X_0}}, J_{|_{X_0}})-\overline{\mu}\right)\Omega_{|_{X_{0}}}^{[n]},
\end{split}
\end{equation}
where for the last equality we used that $\pi: \tstX \to \PP^1$ is a smooth submersion and Ehresmann's fibration lemma.  With a slight abuse of notation, we shall denote from now on by $F_0 := \Omega_{|_{X_0}}$ the symplectic GK structure induced by $\Omega$ on the (necessarily smooth by our hypothesis on $\tstX$) central fiber $(X_0, \poiss_0, \alpha)$, and further denote $h_0:=h_{|_{X_0}}$, $\xi_0 :=\hat\xi_{|_{X_0}}$.  Notice that $J\xi_0$ induces a holomorhic Poisson $S^1$-action on $(X_0, \poiss_0, \alpha_0)$  and $F_0 \in \GK_{\poiss_0, \alpha_0}$ is an $S^1$-invariant compatible GK structure. Furthermore, as $\hat \xi$ is tangent to $X_0$ (and denoted by $\xi_0$ on the central fiber), \eqref{gradient} yields $\xi_0 = J_0F_0^{-1}(dh_0)$.
By \cite[Thm.\ 4.18]{ASU} (and using that the mean of $\Gscal(F_0, J_0)$ equals $\overline{\mu}$,  see \cite[Rem.\ 3.9]{ASU}) we can thus rewrite \eqref{limit} as
\[ \lim_{t\to \infty}\boldsymbol{\tau}_{F_t}(\dot{F}_t)= -\Futc_{\poiss_0,\alpha_0}(\xi_0),\]
where, we recall,  $\Futc_{\poiss_0, \alpha_0}(\xi_0)= \int_X h_0\left(\Gscal(F_0, J_0) -\overline{\mu}\right)dV_{F_0}$ is the Futaki invariant associated to $\GK_{\poiss_0,\alpha_0}^0$.   By Lemma~\ref{l:GK-Futaki=K-Futaki} 
\[ \Futc_{\poiss_0,\alpha_0}(\xi_0) =\Futc_{\alpha_0}(\xi_0),\]
where $\Futc_{\alpha_0}(\xi_0)$ is the usual Futaki invariant of the K\"ahler manifold $(X_0, \alpha_0)$.   It is now a well-established fact (see \cite{dervan-ross,DT,odaka,lahdili2, s-dyrefelt1, wang}) that  ${\rm DF}(\tstX, \tstA)=-\Futc_{\alpha_0}(\xi_0)$. This concludes the proof. 
\end{proof}

\begin{rem}\label{r:slope} In one direction, Conjecture~\ref{c:intro} predicts that if $(X, \poiss, \alpha)$ admits a cscGK structure  $F_0 \in \GK_{\poiss,\alpha}$,  then for any $\poiss$-Poisson K\"ahler test configuration $(\tstX, \Poiss, \tstA)$ of $(X, \poiss, \alpha)$ one has the inequality  ${\rm DF}(\tstX, \tstA) \geq 0$ with equality if and only if $(\tstX, \Poiss, \tstA)$ is Poisson product. The naive idea, which has been implemented in the cscK case by the deep works~\cite{BB, BDL, C, PS}, consists of associating to each test configuration a ray $F_t \in \GK_{\poiss, \alpha}, \, t\in [0, \infty)$ of \emph{geodesics} (see \cite{ASU} for the precise definition of the geodesic condition) emanating from $F_0$. If this were possible and the geodesic ray were smooth, 
the monotonicity property of the Mabuchi $1$-form $\boldsymbol{\tau}$ along smooth geodesics established in \cite[Prop.~5.7]{ASU}  and the fact that $\boldsymbol{\tau}_{F_0}=0$ would imply $\lim_{t\to \infty}\boldsymbol{\tau}_{F_t}(\dot{F}_t) \geq 0$ with possible equality only if $F_t$ is generated by the flow of a real holomorphic vector field on $X$.  Proposition~\ref{p:smooth-ray} will then yield the result. However, even in the cscK case,  one cannot obtain in general a \emph{smooth} geodesic ray as above. Instead, the results in \cite{C,PS} yield the existence of a $C^{1,1}$ geodesic ray whereas the results in \cite{BB, BDL, s-dyrefelt1} extend the definition, the monotonicity property  and the asymptotic at infinity  of $\boldsymbol{\tau}$ for such weak geodesic rays. Neither of these results is as of yet available in the GK setup. Nevertheless, Conjecture~\ref{c:intro} has been confirmed in \cite{ASU} for \emph{toric} Poisson K\"ahler manifolds $(X, \poiss, \alpha, \T)$, where the notion of Poisson K-polystability was replaced by a similar notion of \emph{uniform} (Poisson) K-stability on \emph{toric} test configurations~\cite{do-02}, noting that these are automatically Poisson test configuration for any toric Poisson structure hence the Poisson tensor does not affect the K-stability notion. 
\end{rem}

A  ramification of  Conjecture~\ref{c:intro} in the case when $(X,L)$  is a polarized projective manifold which admits a cscK metric in $2\pi c_1(L)$ (and hence is K-polystable by \cite{BDL,stoppa}) reads as follows:

\begin{conjecture}\label{thm-main}  Let $(X, L)$ be a smooth  polarized projective manifold which admits a cscK metric in $2\pi c_1(L)$ and $\poiss \in H^0(X, \Wedge^2 T^{1,0}_X)$ be a non-zero holomorphic Poisson tensor. Then the following are equivalent:

\begin{enumerate}
\item The bivector $\poiss$ is polystable for the linear action of $\Aut_{\rm red} (X)$ on $H^0(X, \Wedge^2 T^{1,0}_X)$;

\item There exists $\epsilon >0$ such that for $|\lambda|< \epsilon$, $(X, \lambda\poiss, 2\pi c_1(L))$ admits a cscGK structure.
\end{enumerate}
\end{conjecture}

\noindent 
Using the Yau--Tian--Donaldson correspondence, Theorem~\ref{thm:main} in the Introduction establishes the direction $(1) \Rightarrow (2)$ of Conjecture~\ref{thm-main} in the Fano case.  This conjecture represents an initial foray into a much larger question of K-stability for generalized K\"ahler manifolds, on which we make some remarks:

\begin{rem}\label{rmk:caveats}
Regarding the conjecture more generally, several remarks and caveats are in order:
\begin{enumerate}
\item When $H^{2,0}(X, \C)=0$, the existence result of Gualtieri \cite{Gualtieri-Hamiltonian} implies that for any Poisson tensor $\poiss$ and any K\"ahler class $\alpha$,  there exists a GK structure in the class $\alpha$ with Poisson tensor $\lambda\poiss$ for $\lambda$ small enough; the question is thus whether we can find such a structure with constant generalized scalar curvature.
\item The notions of Poisson test configurations, and hence Poisson K-polystability, are invariant under independent resclaings of the class $\alpha$ and the Poisson tensor $\poiss$.  On the other hand, if $F\in \alpha$ is a symplectic GK structure with Poisson tensor $\poiss$, then $t^{-1}F$ is a symplectic GK structure with Poisson tensor $t \poiss$, so the conjecture is only invariant under this particular combination of scalings.
\item Determining the constant $\epsilon$ in the condition (1) of the Conjecture is  an interesting problem specific to GK geometry,  which is challenging even for the simplest examples.

\item Restricting the  above theory to K\"ahler deRham classes $\alpha$ is natural. If $X$ is a compact $2n$-dimensional complex torus, it admits symplectic GK structures $F$ with non-degenerate Poisson tensor $\poiss$ and non-K\"ahler  deRham class $\alpha=[F]$. In this case,  the generalized K\"ahler-Ricci flow starting from $F$ converges  at infinity to a non-degenerate symplectic GK structure $F_0\in \GK_{\poiss, \alpha}$ having zero generalized K\"ahler scalar curvature~\cite{garciafernandez2023nonkahler, streets2012generalized,streets2016pluriclosed}.  This structure is expressed as 
$F_{0}=\omega_{HK} + \beta + \bar \beta \in \alpha$, where $\omega_{HK}$ is the unique HK metric in the $(1,1)$ part of $\alpha$ in the Hodge decomposition $H^2(X, \C)=H^{1,1}(X, \C) \oplus H^{2,0}(X, \C) \oplus H^{0,2}(X, \C)$ whereas the holomorphic $(2,0)$-form $\beta$ is determined by  the $(2,0)$-part of $\alpha$ in that decomposition. Furthermore, $F_0, \beta$ and $\poiss$ are related by the identity
${\rm Im}(\beta) =\frac{1}{4} F_0\left({\rm Re}(\poiss) \right)F_0$, proving that $t\poiss, \, t\neq 1$ cannot be realized as a Poisson tensor of a symplectic GK structure in the given deRham class $\alpha$.

\item The notion  of Poisson K-polystability in the conjecture should be taken as a guiding principle, rather than the definitive statement. Even for classical K\"ahler manifolds, the YTD conjecture remains open beyond the K\"ahler--Einstein case and there are currently different versions of a correspondence as above, obtained by strengthening the K-polystability condition in various ways: e.g. by testing stability on more general objects than test configurations~\cite{boucksom-et-al, ChiLi1, piccone}, by requiring the positivity of a family of numerical invariants associated to each test configuration~\cite{DZ}, or by requiring a \emph{uniform} positivity of the Donaldson--Futaki invariant with respect to a suitable `norm' on test configurations~\cite{CC, H,ChiLi,  sz-PhD, trusiani}.  In each case, it is conceivable to explore the above conjecture under similar modifications of Poisson K-polystability.  
\item As emphasized in the K\"ahler--Einstein case in \cite{LWX, odaka1} and in cscK case in \cite{dn, fs}, one can hope to use the cscGK structures in $\GK_{\poiss, 2\pi c_1(L)}$ as an \emph{analytic} stability condition, allowing for the construction of an analytic moduli space of polarized Poisson varieties  $(X, \poiss, L)$ admitting such structures.  On the other hand, since the notion of Poisson K-stability is algebraic, one can also expect to produce reasonable moduli spaces purely algebro-geometrically.
\end{enumerate}
\end{rem}

\section{Kempf--Ness theory for tame symplectic forms}\label{a:GIT}

The classical Kempf--Ness theorem concerns the relation between stability questions in geometric invariant theory (GIT) for complex varieties, and momentum maps for an associated compatible symplectic structure.  In the proof of our main theorem, we will need to invoke some aspects of this theory under the weaker assumption that the symplectic form tames the complex structure, i.e.~that the complex structure and symplectic form together still define a positive definite metric, but the almost complex structure need no longer be orthogonal.  {This is a simpler case of the setup in \cite[Sec.5]{delloque}, where the symplectic forms come in families (cf. also \cite{dms,ortu}). For convenience, we explain below how the results we need carry over to the tame setting.}  

Consider the complex vector space $\C^k$ endowed with its standard flat K\"ahler structure $(g_0, J_0, \omega_0)$ and  fix a compact subgroup $K \subset U(k)$.  Let  $U=B(0,\varepsilon)$ be an open ball centered at the origin, and let $\omega$ be a $K$-invariant symplectic form $\omega$ of regularity $C^r$ for some $r \geq 2$.
We assume that
\begin{enumerate}
\item  The symplectic form $\omega$ tames the standard complex structure $J_0$, agrees with $\omega_0$ at the origin, and there exists $\Lambda > 0$ so that
\begin{align*}
1/\Lambda < ||\omega^{1,1}||_{g_0} \leq ||\omega||_{g_0}<\Lambda
\end{align*}
holds on $U$.
\item The action of $K$ on $(U, \omega)$ is Hamiltonian with momentum map $\mu_{\omega}^K$ such that $\mu_{\omega}^K(0)=0$.
\end{enumerate}
We fix an ${\rm ad}$-invariant Euclidean inner product $\langle \cdot, \cdot \rangle$ on the compact Lie algebra $\mathfrak{k}$ of $K$. The choice of $\langle \cdot, \cdot \rangle$ gives rise to an identification $\sharp : \mathfrak{k}^* \to \mathfrak{k}$. 
We further  let $K^{\C}$  denote the complexification of $K$, hence $K^{\C}$ is a reductive complex subgroup of $GL(k,\C)$ and we adopt the usual definition for stable points of $\C^k$ with respect to the $K^{\C}$-action, see Definition~\ref{d:GIT}.  Our goal in this subsection is to establish the following result, which can also be derived from the results in \cite[Sec. 5]{delloque}:

\begin{prop}\label{p:local-tamed-GIT} There exists an open ball $U_0=B(0, \varepsilon_0) \subset U$ such that if $p \in U_0$ is a $K^{\C}$-polystable point,  then $K^{\C}\cdot p$ contains a point $q\in U$ for which $(\mu^K_{\omega}(q))^{\sharp}$ belongs to the Lie sub-algebra of the stabilizer of $q$ in $K$. If,  furthermore,  $p \in U_0$ is a stable point, then $\mu_{\omega}^K(q)=0$.
\end{prop} 

This is a generalization of part of the classical Kempf--Ness theorem (the case when $\omega\equiv \omega_0$) to our tamed setting.  Our approach is to adapt the arguments in \cite{GRS} from the compatible to the tamed case, which rely on the \emph{momentum flow}: 
\begin{defn} A one-parameter family $p(t) \in U$ is a solution of the \emph{momentum flow} if
\begin{equation}\label{moment-flow}
 \frac{d}{dt}p(t) = J_0 V_{(\mu_{\omega}^K(p(t)))^{\sharp}}, \qquad p(0)=p_0,\end{equation}
 where for any $\xi \in \mathfrak{k}$, $V_{\xi}$ denotes the fundamental vector field on $\C^k$.
\end{defn}

A key formal property from the compatible case extends naturally here, namely that solutions of the momentum flow preserve complexified orbits:

\begin{lemma}\label{l:orbit} Let $p(t)$ be the solution of the momentum flow and $\tau(t) \in K^{\C}$ the solution of the ODE
\[ \frac{d}{dt} \tau(t) = (L_{\tau(t)})_*\left(\i \mu(p(t))^{\sharp}\right), \qquad \tau(0)={\rm Id}_{K^{\C}}. \]
Then $p(t)= \tau(t) \cdot p_0$. In particular, the momentum flow stays in $K^{\C}\cdot p_0$.
\end{lemma}
\begin{proof} Elementary, see \cite[Lemma 3.2]{GRS} for details.
\end{proof}

The remainder of the proof consists of an ODE analysis of the momentum flow lines.  A standard feature of the compatible case is that the momentum flow is the gradient flow for the square norm of the moment map, which we generalize below in the tamed case. Let   
\[ f(p) = \frac{1}{2}||\mu_{\omega}^K||^2 := \tfrac{1}{2} \langle \mu^K_\omega(p)^{\sharp}, \mu^K_\omega(p)^{\sharp} \rangle. \] 
Hence by the definition of momentum map, for any vector field $v$ on $\C^k$ and an orthonormal basis $\{e_1, \ldots, e_{\ell}\}$ of $(\mathfrak{k}, \langle \cdot, \cdot \rangle)$,
\[ 
\begin{split}
df(v)_p & = \sum_{i=1}^\ell \left(d(\langle \mu^K_{\omega}, e_i \rangle)(v) \langle \mu^K_{\omega}, e_i \rangle\right)(p)  \\
&=- \sum_{i=1}^{\ell} \left(\omega(V_{e_i}, v)\langle \mu^K_{\omega}, e_i \rangle\right)(p)  \\
&=- \omega_p(V_{\mu^K_{\omega}(p)^{\sharp}}, v).
\end{split}
\]
It thus follows that for each $p\in U$:
\[ \omega^{-1}(df)(p)= -V_{\mu(p)^{\sharp}}(p), \]
showing that the momentum flow equivalently satisfies
\begin{equation}\label{almost-gradient} \frac{d}{dt}p(t) = - J_0 \omega^{-1} (df)(p(t)).
\end{equation}
In the case when $\omega$ is $J_0$-compatible, $-J_0 \omega^{-1} (df)= - g^{-1}(df)$,  so by \eqref{almost-gradient},  the momentum flow  becomes up to a sign the \emph{gradient flow} for $f$ for the corresponding Riemannian metric $g= -\omega J_0$. In general, when $\omega$ only tames $J_0$, it still defines a Riemannian metric $g=-(\omega^{1,1}) J_0$ on $U$ and we observe a similar phenomenon:
\begin{lemma}\label{l:symplectic-gradient} Let $p(t)$ be the solution of the momentum flow. Then
\[ \frac{d}{dt} f(p(t)) = -\left\|\omega^{-1}(df)\right\|^2_{g} (p(t)). \]
\end{lemma}
\begin{proof} By \eqref{almost-gradient}, 
\[ \frac{d}{dt} f(p(t))= -\left(df(J_0\omega^{-1} (df))\right)(p(t))= -\omega(\omega^{-1}(df), J_0 \omega^{-1}(df))(p(t))=-||\omega^{-1}(df)||^2_{g}(p(t)).\]
\end{proof}

Next, we show that a \emph{Lojasiewicz gradient estimate}~\cite{Loja} holds for $f$ (cf. \cite[Thm.6.4 \& Prop.6.8]{BM} and \cite[Thm.1]{Feehan} for the statement below).

\begin{lemma} \label{l:Loja} After possibly shrinking $U$, there exist constants $\delta>0, \, C>0, \, \lambda \in (1/2, 1)$ such that if $a$ is a critical value of $f$  and $p \in U$ is any point, then 
\begin{equation}\label{e:gradient} 
|f(p) - a|<\delta \, \, \Longrightarrow \, \,  |f(p) - a|^{\lambda} < C\|(\omega^{-1} df)(p)\|_g. \end{equation} 

\begin{proof}
By the relative equivariant Moser lemma we obtain, after possibly shrinking $U$, that $(U, \omega, K, \mu_{\omega}^K)$  is $K$-equivariantly symplectomorphic to $(\widetilde U, \omega_0, K, \mu_0^K)$ where \[ \mu_0^K(z)= \tfrac{1}{2} {\rm pr}_{\mathfrak{k}^*}(z \otimes \bar z)\]
is the standard $K$ momentum map of $\omega_0$.
The relevant estimate holds on $\widetilde U$ for the real analytic function $f_0 := \frac{1}{2}\langle\mu_0^K, \mu_0^K \rangle$ (see e.g. \cite[Thm.1]{Feehan}), and since $f$ is identified with $f_0$ via a $C^{r-1}$ diffeomorphism, after possibly further shrinking $U$, the claim follows.
\end{proof}
\end{lemma}

\begin{lemma}\label{l:stability}There exists an open ball $U_0=B(0, \varepsilon_0)$  such that for any $p_0 \in U_0$ the solution $p_t$ of \eqref{moment-flow} with initial value $p_0$ is global and stays in a closed ball $\overline{B(0, \varepsilon_1)} \subset U$.
\end{lemma}
\begin{proof} 
We apply \eqref{e:gradient} with $a=0$, noting  that $f(0)=0$ is a global minimum of $f$.  As $f$ is continuous, we let $B(0, \epsilon') \subset U$ be an open ball centered at $0$,  such that $f(p)<\delta'<\delta$ on $B(0, \epsilon')$. Let $p(t)$  be a solution of the momentum flow starting at $p_0 \in B(0, \epsilon')$. As $f(p(t))$ decreases (see Lemma~\ref{l:symplectic-gradient}), we have $0\leq f(p(t))<\delta'$. Using \eqref{e:gradient} and Lemma~\ref{l:symplectic-gradient},  we evaluate for any $t>0$:
\begin{equation}\label{e:Lgradient} -\frac{d}{dt} \left( f(p(t))^{1-\lambda} \right)= (1-\lambda)f(p(t))^{-\lambda}||\omega^{-1}(df)||^2_g(p(t)) \geq \left(\frac{1-\lambda}{C}\right)||\dot p (t)||_g.\end{equation}
Integrating between $[0, T]$, we get
\[d^g(p_0, p(T)) \leq  L^g(p(t)) \leq \left(\frac{C}{1-\lambda}\right)\left(f(p_0)^{1-\lambda} - f(p(T))^{1-\lambda}\right) < (\delta')^{1-\lambda}\left(\frac{C}{1-\lambda}\right),\]
and hence 
\[d^g(0, p(T)) \leq \epsilon' + (\delta')^{1-\lambda}\left(\frac{C}{1-\lambda}\right).\]
The claim follows by taking $\delta'$ and $\epsilon'$ small enough,  and noting that we have $1/\Lambda g_0 \leq g \leq  \Lambda g_0$ on $U$. \end{proof}

The key property of the flow  we need to establish is the following convergence result:
\begin{lemma} \label{flow-estimates} If $p_0 \in U_0$, then the limit $p_{\infty} = \lim_{t\to \infty} p(t)$ exists in $U$ and satisfies  $V_{(\mu^K_{\omega}(p_{\infty}))^{\sharp}}(p_{\infty})=0$.  Furthermore, there are uniform positive constants $C, \epsilon, T$ such that
\[ d^g(p_{\infty}, p(t))  \leq \frac{c}{(t-T)^{\epsilon}},  \, \, \, \forall  t>T.\]
\end{lemma}
\begin{proof} We follow the proof \cite[Theorem~3.3]{GRS}.  Let $p(t)$ be a solution of the momentum flow and denote
\[ a:= \lim_{t\to \infty} f(p(t)).\]
This limit exists as $f(p(t))$ decreases (see Lemma~\ref{l:symplectic-gradient}) and is non-negative. Notice that $a$ is a critical value of $f$: this follows from \eqref{almost-gradient}. Thus, $a=f(p_{\infty})$ is a critical value of $f$.  Let $T>0$ be such that 
\[ a< f(p(t)) < a + \delta \qquad  \forall t\geq T. \]
It follows  from \eqref{e:Lgradient} that for any $t\geq T$
\begin{equation}\label{key-estimate}
 \int_{t}^{\infty}\|\dot p(t)\| dt  \leq \left(\frac{C}{1-\lambda}\right)\left(f(p(t))-a\right)^{1-\lambda}.\end{equation}
This shows that the whole flow line $p(t)$ converges to a point $p_{\infty}$.

Letting further
\[ \rho(t):= \left(f(p(t)) -a\right)^{-2\lambda + 1}\]
we estimate similarly
\[ \dot \rho(t)=(2\lambda-1)\left(f(p(t)) -a\right)^{-2\lambda}\| \omega^{-1}(df)\|^2_g \geq\left(\frac{2\lambda-1}{C^2} \right) \qquad \forall t\geq T.  \]
This implies
\[\rho(t) \geq \left(\frac{2\lambda-1}{C^2}\right)(t-T) \qquad \forall t\geq T  \] 
and thus
\[ (f(p(t)) -a)^{1-\lambda} = \rho(t)^{-\frac{1-\lambda}{2\lambda-1}} \leq \left(\left(\frac{2\lambda-1}{C^2}\right)(t-T)\right)^{-\frac{1-\lambda}{2\lambda-1}} . \] 
Combined with \eqref{key-estimate}, this concludes the proof.
\end{proof}

\begin{proof}[Proof of Proposition~\ref{p:local-tamed-GIT}] By Lemmas~\ref{l:stability} and \ref{flow-estimates},  the momentum flow $p(t)$ starting at $p_0 \in U_0$ is global and converges to $p_{\infty} \in U$; by Lemma~\ref{l:orbit}, there exists a corresponding solution $\tau(t)\in K^{\C}$  such that $p(t)= \tau(t)\cdot p_0$. As $p_0$ is polystable by assumption, $K^\C \cdot p_0$ is closed and thus $p_{\infty} \in K^{\C} \cdot p_0 \cap U$. If $p_0$ is,  furthermore,  $K^{\C}$-stable, $p_{\infty} \in K^{\C} \cdot p_0$ has a trivial stabilizer, and thus $\mu(p_{\infty})^{\sharp}=0$.
\end{proof}
\begin{rem} Thinking of $\C^k$ as an affine chart of $\PP^{k}_{\C}$, an easy  modification of the above arguments shows that for any compact subgroup $\hat K < PU(k+1)$ and any $\hat K$-invariant symplectic  $2$-form $\omega$ on $\PP^k_{\C}$ with momentum map $\mu^{\hat K}_{\omega}$,  such that $\omega$   tames the standard complex structure of $\PP^k_{\C}$, one direction of the Kempf--Ness correspondence holds.  Namely,  if $p\in \PP^k_{\C}$ is a $\hat K^{\C}$-polystable (resp. stable)  point then $||(\mu_{\omega}^{\hat K})^{\sharp}||^2$ has a  critical point (resp. zero) on the orbit $\hat K^{\C} \cdot p$. It is a natural question to ask if the $\hat K^{\C}$-polystability (resp. stability) of $p$ is also a necessary condition  for the vanishing of $d|| (\mu^{\hat K}_{\omega})^{\sharp}||^2$ (resp. of $\mu^{\hat K}_{\omega}$) on $\hat K^{\C} \cdot p$, i.e. whether the full Kempf--Ness correspondence can be extended from the compatible to the tamed setup.
\end{rem}

\section{cscGK structures from Poisson deformations}

In this section we establish Theorem \ref{thm:main}.  The main point is to exploit an idea from \cite{sz}, with supplements in \cite{CS} and \cite{ortu}, where it was established that the existence of cscK metrics is stable under \emph{polystable} small complex deformations of a given cscK polarized manifold.  In our case instead of varying the complex structure we shall vary the Poisson tensor.  We refer to the introduction for an explanation of the broad strategy of the proof.  To begin we recall various definitions and set notation used throughout this section:\\

\begin{enumerate}
\item $X= (M, J_0)$ is a smooth compact complex manifold which admits a cscK metric $(\omega_0,g_0)$.  We assume $H^{2,0}(X, \C)=\{0\}$ through this section.  Of course, this is automatic in the Fano case, but most of our arguments apply in the general cscK case.\\

\item $K$ is the group of K\"ahler isometries of $(X, J_0, \omega_0)$ inside the reduced group of complex automorphisms $\Aut_{\rm red}(X)$ of $(X, J_0)$.  In the Fano case $\Aut_{\rm red}(X)=\Aut_{\circ}(X)$ is just the connected component of the automorphism group of $X$. By the Lichnerowicz--Matsushima theorem, $\G:=K_{\C}=\Aut_{\rm red}(X)$ is reductive as $X$ admits a cscK metric.\\

\item  $\Pi \subset H^0(X, \wedge^2 T^{1,0}_X)$ is the algebraic subset of holomorphic Poisson tensors  inside the vector space  $H^0(X, \wedge^2 T^{1,0}_X)$ of holomorphic bi-vectors; we assume $\Pi \neq \{0\}$. \\ 
\item  $W := {\rm span}_{\C} \Pi < H^0(X, \wedge^2 T^{1,0}_X)$ denotes the linear subspace spanned by Poisson bi-vectors, also called  the \emph{Zariski tangent space} of $\Pi$.  Observe that the group $\G$ acts linearly on $H^0(X, \wedge^2 T^{1,0}_X)$ by pushforward, preserving $\Pi$ and hence $W$. It thus follows that the stability properties  of $\poiss \in \Pi$  with respect to linear actions of $\G$ on the vector spaces $H^0(X, \wedge^2 T^{1,0}_X)$  and on $W$ agree, so we shall work without loss of generality on $W$.\\

\item Let $\AGK_{\omega_0}$ denote the infinite dimensional formal Fr\'echet manifold of almost-complex structures $J$ on $X$ tamed by the symplectic form $\omega_0$, endowed with the Fr\'echet topology of $C^{\infty}(X)$ convergence of smooth tensors on $X$,  and  with the formal K\"ahler structure $(\boldsymbol{\Omega}, {\bf J})$ coming from the space of $\omega_0$-compatible almost GK structures.  The group of Hamiltonian transformations $\Ham(X, \omega_0)$ acts on $(\AGK_{\omega_0},\boldsymbol{\Omega}, {\bf J})$ in a Hamiltonian way, with momentum map identified through the global $L^2$ product on $(X, \omega_0)$ with the normalized version of Goto's scalar curvature function~\cite[Def.5.3]{Goto-moment} 
\begin{align*}
\mathring{\Goto}(\omega_0, J):=\left(\Goto(\omega_0,J)- \overline{\mu}\right)
\end{align*}
where $\overline{\mu}$ is a topological constant  making $\mathring{\Goto}(\omega_0,J)$ of zero mean with respect to $\omega_0$.  
Observe here that as $K\subset \Ham(M, \omega_0)$, $K$ also acts on $(\AGK_{\omega_0},\boldsymbol{\Omega}, {\bf J})$ in a Hamiltonian way, with momentum map
\[ \langle \boldsymbol{\mu}(J), \xi \rangle = -\int_X \mathring{\Goto}(\omega_0,J) h^{\xi}_{\omega_0} \omega_0^{[n]},  \] 
where $h^{\xi}_{\omega_0}$ is a  Hamiltonian  potential  of zero mean with respect to $\omega_0$ of an element $\xi \in \mathfrak{k}:={\rm Lie}(K)$. If, furthermore, $J \in \AGK_{\omega_0}$ corresponds to a symplectic GK structure, $\Goto(\omega_0, J)= \Gscal(\omega_0, J)$ where $\Gscal(\omega_0, J)$ is given by \eqref{eq:Gscal}.\\

\item We endow the space $\GK_{\poiss, \alpha}$ introduced in Definition~\ref{d:GK-class}  with the $C^{\infty}(X)$ Fr\'echet topology, or the H\"older $C^{k+\lambda}(X)$ topology for $k$ sufficiently large and $\lambda \in (0,1)$. In the latter  case,  we shall denote that space by $\GK^{k+\lambda}_{\poiss, \alpha}$ and by $\K_{\alpha}^{k+\lambda}$ its subspace of K\"ahler forms in $\alpha$ of regularity $C^{k+\lambda}(X)$. Similarly, $\AGK_{\omega_0}^{k+\lambda}$ will denote the space of almost complex structures $J$ of regularity $C^{k+\lambda}(X)$ tamed by the symplectic $2$-form $\omega_0$.
\end{enumerate}

\subsection{An extended Gualtieri map}\label{sec:gualtieri-map}

In this first step,  for a $\poiss \in W := {\rm span}_{\C}\Pi \subset H^0(X, \Wedge^2 T^{1,0}_X)$ of suitable (small) size,  and an initial K\"ahler metric $\omega_{\varphi} = \omega_0 + dd^c \varphi$ in the space $\K_{\alpha}$ of $J_0$-compatible K\"ahler metrics in the deRham class $\alpha$, we shall assign a $2$-form $F_{\sigma, \varphi}$ which tames the complex structure $J_0$.  Furthermore, in the case when $\sigma \in \Pi$ is Poisson, the form $F_{\sigma, \varphi}$ will be closed and belong to $\GK_{\sigma, \alpha}$.
This is essentially the construction in \cite{Gualtieri-Hamiltonian}  but extended to non-Poisson bi-vectors in $W$.  As we shall not vary the complex structure $J_0$ at this step, through this subsection we shall denote $J=J_0$ the underlying complex structure of $X$.

For any $\omega \in \K_{\alpha}$, we  write $\omega=\omega_{\varphi}:= \omega_0 + dd^c \varphi$ where $\varphi \in \mathcal{H}_{\omega_0} \subset C^{\infty}(X)/\R$ is in the Fr\'echet space  $\mathcal{H}_{\omega_0}$ of  $\omega_0$-relative K\"ahler potentials.  Note that we will work with open balls in the vector spaces $W$ and $C^{k + 2 +\lambda}(X, \mathbb R)$.  As the norms on these spaces are defined using $\omega_0$, which is $K$-invariant, the balls themselves are automatically $K$-invariant.  

\begin{lemma}\label{l:gualtieri} Given the setup above and $k \in \mathbb N$, there exist $\delta_1, \delta_2 > 0$ defining balls $U_1  = B(0,\delta_1) \subset W$ and $U_2 = B(0,\delta_2) \subset C^{k + 2 +\lambda}(X)$ and a $K$-equivariant map
\begin{align*}
\boldsymbol{\Phi} : U_1 \times U_2 \to C^{k +\lambda}(X, \Wedge^2 T^*X)
\end{align*}
satisfying 
\begin{enumerate}
\item  $\boldsymbol{\Phi}(0, \varphi)=\omega_{\varphi}$ is a K\"ahler metric in $\K^{k+\lambda}_\alpha$ for any $\varphi \in U_2$;
\item $\boldsymbol{\Phi}(\sigma, {\varphi})$  is a generalized K\"ahler metric in $\GK^{k+\lambda}_{\sigma, \alpha}$ for any $\sigma \in U_1 \cap \Pi$ and any ${\varphi} \in U_2$;
\item $\boldsymbol{\Phi}(\sigma, \varphi)$ tames $J$ for any $(\sigma, \varphi) \in U_1 \times U_2$.
\end{enumerate}
\end{lemma}

\begin{proof} We recall the construction in the case when $\poiss\in \Pi$ is a Poisson bi-vector following \cite[Corollary 7.3]{Gualtieri-Hamiltonian}.  We let $Q:={\rm Re}(\poiss)$ be the real $2$-vector tensor corresponding to $\poiss$. By \cite{Gualtieri-Poisson}, $F\in \GK_{\poiss, \alpha}$ if and only if $F$ is closed, tames $J$ and
\begin{equation}\label{main} F Q F = FJ -J^\sharp F. \end{equation}
In the above equality and in what follows we use $J^{\sharp}$ to denote  $-J^*$.  The reason for this is that if $g$ is a $J$-compatible Riemannian metric, then $J^{\sharp} = g J g^{-1}$ is the Riemannian dual of $J$.

We seek $F(t) \in \GK^{k+\lambda}_{t\poiss, \alpha}, \, t\in [-1,1]$, analytic in $t$, constructed as a power series
$$F(t) =F_0 + tF_1 + \ldots,$$
where $F_0:=\omega_{\varphi}$ and each $F_n$ is a $d$-exact real 2-form.  Observe that the equation
\begin{equation}\label{B condition}
F(t) J - J^\sharp F(t)- tF(t) Q F(t)=0
\end{equation}
is equivalently expressed as
\begin{equation}\label{e1}
2J^\sharp   (F(t)^{2,0 + 0,2} ) + tF(t)  Q  F(t) = 0,
\end{equation}
where the complex types are with respect to $J$.
If we decompose \eqref{e1} term-by-term with respect to powers of $t$, we have (factoring out $t^n$)
\begin{equation}\label{term condition}
J^\sharp F_0^{2,0+0,2} = 0, \qquad J^\sharp   F_n^{2,0+0,2} =- \frac{1}{2} \sum_{i+j = n-1} F_i  Q  F_j.
\end{equation}
Since $F_0 = \omega_{\varphi}$ is $(1,1)$, this is satisfied for $n=0$.  Given $F_i$ for all $i < n$, \eqref{term condition} fixes $F_n^{2,0+0,2}$.  We will build $F_n^{1,1}$ (and hence $F_n$)  by induction. 

First observe that for $n=1$,   \eqref{term condition} is equivalent to 
\begin{equation}\label{F1-(0,2)} -4\sqrt{-1}F_1^{0,2} =  F_0 \poiss F_0= \omega \poiss \omega, \end{equation}
which yields $\bar\partial F_1^{0,2}=0$ by using  $d\omega=0$ and $\bar \partial \poiss =0$.  Since we assume that $H^{2,0}(X)=0$, there exists a $(0,1)$-form $\beta_1$ such that 
  \[ F_1^{0,2}= \bar \del \beta_1.\]
 Given such a $1$-form $\beta_1$, we  let  
 \[ F_1^{1,1} = \partial \beta_1  + \bar{\partial} \bar{\beta}_1 \quad \longleftrightarrow \quad F_1 := d(\beta_1 + \bar\beta_1).  \]
Now assume there exist $F_0, F_1, \ldots, F_{n-1}$ satisfying \eqref{term condition} and moreover $F_k= d(\beta_k + \bar \beta_k), 1 \leq k \leq n-1$. We claim
\begin{enumerate}
    \item $\sum_{i+j = n-1} \left(F_i  Q  F_j\right)^{1,1} = 0$,\\
    \item $\bar\del\left(\sum_{i+j = n-1} \left(F_i  Q  F_j\right)^{0,2}\right) = 0$.
\end{enumerate}
Given this claim, we may define $F_n^{2,0+0,2}$ via \eqref{term condition}, solve $F_n^{0,2} = \bar \del \beta_n$ and let $F_n = d(\beta_n + \bar \beta_n)$, finishing the induction.  To show the first claim observe that it is equivalent to
$$\sum_{i+j = n-1} \left(J^\sharp F_i  Q  F_j + F_i  Q  F_j  J\right)=0.$$
Since $Q$ anti-commutes with $J$, i.e., $J  Q = -Q  J^\sharp$, the LHS can be rewritten as
\begin{align*}
 \sum_{i+j = n-1} &\left(J^\sharp F_i  Q  F_j -  F_i  J  Q   F_j - F_i  Q   J^\sharp  F_j + F_i  Q   F_j  J \right) \\
=& \sum_{i+j = n-1} (J^\sharp  F_i - F_i  J) Q F_j - \sum_{i+j = n-1} F_i Q (J^\sharp F_j - F_j  J).
\end{align*}
As \eqref{term condition} holds for all $n' < n$ by induction, we make the substitution
$$J^\sharp F_i - F_i  J = \sum_{k+l = i-1} F_k  Q  F_l$$
(and likewise for $F_j$), to obtain 
\begin{align*}
\sum_{i+j = n-1} \left(J^\sharp  F_i   Q  F_j + F_i  Q   F_j   J\right)
&= \sum_{j+k+l = n-1} F_k  Q  F_l  Q  F_j \;-\; \sum_{i+k+l = n-1} F_i  Q  F_k  Q   F_l =0.
\end{align*}
We note that this inductive proof of claim (1) holds independently of claim (2).  The proof of claim (2) is the subtle part of the argument in \cite{Gualtieri-Hamiltonian},  where the fact that $\poiss \in \Pi$ will be used.  To this end, Theorem 5.3, Proposition 3.5 and Remark 3.6-(1) in \cite{Gualtieri-Hamiltonian} show that the equation \eqref{B condition} we are solving can be equivalently rewritten in the form
\[ \psi(F)^{0,2} =0, \qquad   \psi(F):= (1+ tF\poiss_0)^{-1}F, \quad \poiss_0 := \frac{\sqrt{-1}}{4} \poiss,\]
and we have assumed that $t$ is small enough so that  the inverse is defined. From this equivalent point of view, and using  that $\psi(F)$ satisfies a Maurer--Cartan equation  with respect to the derived Koszul bracket of the holomorphic Poisson tensor $\poiss$, it becomes clear (see (66)-(67) in \cite{Gualtieri-Hamiltonian}) that claim (2) holds.

Hence we have built a formal power series for real $d$-closed forms $F(t)\in [\omega]=\alpha$, which satisfy \eqref{B condition}.  The next step is to define the map $\Phi$ for general sections $\sigma \notin \Pi$.  To achieve this we first make an important reformulation of the inductive construction using Hodge theory, which will also be key in proving the convergence.
Let
  \[ \bar \square_{\omega_0} = \bar\del \bar\del^*_{\omega_0} + \bar\del^*_{\omega_0}\bar\del = \tfrac{1}{2}\Delta_{\omega_0} = \tfrac{1}{2}\left (d\delta_{\omega_0} + \delta_{\omega_0} d \right), \]
the $\bar \partial$-Laplacian acting on smooth $(0,2)$-forms.  As $H^{0,2}(X, \C)=0$, this operator is invertible, with inverse denoted by $\overline{{\mathbb G}}_{\omega_0}$ (thus $2\overline{\mathbb{G}}_{\omega_0} = \mathbb{G}_{\omega_0}$  where $\mathbb{G}_{\omega_0}$ is  the inverse of $\Delta_{\omega_0}$,  restricted to $(0,2)$-forms.)  We observe that by construction we have
$$\beta_n := \bar \del^*_{\omega_0} \overline{\mathbb G}_{\omega_0}(F_n^{0,2}),$$
hence
\begin{equation}\label{Fn}
F_n^{1,1} = {\rm Re} \left( \del\bar \del^*_{\omega_0} \overline{\mathbb G}_{\omega_0}(F_n^{0,2})\right) \quad  
 \longleftrightarrow \quad F_n = F_n^{0,2} + {\rm Re} \left( \del\bar \del^*_{\omega_0}  \overline{\mathbb G}_{\omega_0} (F_n^{0,2})\right).\end{equation}
Crucially, if we choose now $\sigma \notin \Pi$ a bi-vector in $W$, we can use the relation \eqref{term condition} (equivalently the first claim of the inductive construction above) to inductively determine $F_n^{0,2}(t)$, and hence $F_n(t)$ by \eqref{Fn}.  Of course, the second claim of the inductive construction above will no longer hold and so we do not expect such $F$ to be closed.

It remains to be shown that this series has a positive radius of convergence, following standard arguments using Hodge theory as in \cite{KM}.  Schauder estimates for the Laplacian and the inductive definition imply that in $C^{k+\lambda}(X, \Wedge^2 T^*X)$ (for given $k\ge 2, 0<\lambda<1$)
\begin{equation*}
\|F_n\|_{k+\lambda} \leq C_{k+\lambda} \|\poiss\|_{k+ \lambda}\left(\sum_{i+j=n-1} \|F_i\|_{k+\lambda} \|F_j\|_{k+\lambda}\right),
\end{equation*}
for some positive constant $C_{k+\lambda}$. We can conclude as in \cite[Chapter 4, Thm.~2.1]{KM} that the power series $F(t) = \sum_{n=0}^{\infty} F_n t^n$ converges  in $C^{k+\lambda}(X, \Wedge^2T^*X)$ for $t \in [0, 1/b_{k+\lambda})$ where $b_{k+\lambda}=16C_{k+\lambda}||\omega_{\varphi}||^2_{k+\lambda}||\poiss||_{k+\lambda}$.  Restricting to a suitable open neighbourhood $(0, 0) \in  U_1 \times U_2   \subset W \times C^{k+2+\lambda}(X)$, we can assume $b_{k+\lambda}<1$.  We thus  define
\[ \boldsymbol{\Phi}(\poiss, \varphi) := F(1).\]
The claimed properties are built into the construction, and moreover the $K$-equivariance  of $\boldsymbol{\Phi}$ follows by the $K$-invariance of $\omega_0$ and the naturality of the construction.
\end{proof}

Observe that Lemma~\ref{l:gualtieri} produces $2$-forms of regularity  $C^{k+\lambda}$, and the radius of convergence of the formal power series depends on $k$.  In the next lemma we employ further elliptic regularity arguments to improve this to smooth forms, again following ideas from \cite{ABD,KM}.

\begin{lemma}[Smooth regularity of ${\bf \Phi}$]\label{l:smooth} In the setup of Lemma~\ref{l:gualtieri}, and possibly shrinking $\delta_1$ and $\delta_2$, if $(\sigma, \varphi) \in U_1 \times U_2$ and $\varphi$ is smooth, then ${\bf \Phi}(t\sigma, \varphi)$ is smooth.
\end{lemma}
\begin{proof} We fix $\varphi\in U_2 \cap C^{\infty}(X)$, and use elliptic regularity as in \cite{KM} to show that ${\bf \Phi}(t\poiss, \varphi) \in C^{\infty}(X, \Wedge^2 T^*X)$.  Let $F_{\sigma, {\varphi}}(t)$ denote the convergent power series constructed in the proof of Lemma~\ref{l:gualtieri} with
\[ {\bf \Phi}(\sigma, \varphi):=F_{\sigma, \varphi}(1). \]
By the inductive construction of $F_{\sigma, \varphi}(t)$, it follows that
\begin{equation} \label{poisson-scale} F_{t\sigma, \varphi}(1)=  F_{\sigma, \varphi}(t), \end{equation}
and
\begin{equation}\label{e2}
(F_{\sigma, \varphi}(t))^{1,1} = \omega_{\varphi} + 2 {\rm Re}\Big(\del\bar \del^*_{\omega_0} \overline{\mathbb G}_{\omega_0}\left( F_{\sigma, \varphi}(t)^{0,2}\right)\Big), \qquad F_{\sigma, \varphi}(t)^{0,2} = \bar \square_{\omega_0}\left( \overline{\mathbb G}_{\omega_0}\left( F_{\sigma, \varphi}(t)^{0,2}\right)\right).
\end{equation}
Furthermore $F_{\sigma, \varphi}(t)$ satisfies \eqref{e1}. Letting 
\[ \Theta:= \overline{\mathbb G}_{\omega_0} \left(F_{\sigma, \varphi}(t)^{0,2}\right)=\overline{\mathbb G}_{\omega_0} \left(F_{t\sigma, \varphi}(1)^{0,2}\right) ,\]
we get
\begin{equation}\label{F/Theta}
F_{\sigma, \varphi}(t) = \omega_{\varphi} + 2 {\rm Re} (\partial \bar \partial^*_{\omega_0} \Theta) + 2 {\rm Re} \left(\bar\square_{\omega_0} \Theta \right).
\end{equation}
Substituting \eqref{F/Theta} into \eqref{e1}, we get
\begin{align}\label{e3}
E_{t\sigma, \varphi}(\Theta) =0,
\end{align}
where 
\[ E_{\sigma, \varphi} (\Theta):= 2\sqrt{-1} \bar \square_{\omega_0}\Theta -\Big(\big(\bar\square_{\omega_0}  +2  {\rm Re} \del \bar \del^*_{\omega_0})(\Theta) + \omega_{\varphi} \big) Q \big((\bar\square_{\omega_0}  +2 {\rm Re} \del \bar \del^*_{\omega_0}) (\Theta) + \omega_{\varphi}\big)\Big)^{0,2}.\]
Note that since $\varphi$ is smooth, $E_{\sigma, \varphi}$ is a non-linear second-order differential operator with smooth coefficients acting on sections  $\Theta$ of $C^{k+\lambda}(X, \Wedge^{0,2}(T^*X \otimes \C))$.  A type decomposition argument with respect to $J$ yields that the linearization of $E_{\sigma, \varphi}$ at $\Theta=0$ is
\[ (D_0 E_{\sigma, \varphi}) (\dot \Theta) = 2\sqrt{-1} \bar \square_{\omega_0} \dot \Theta.\]
It follows that for any $\Theta$ close to $0$ in the $C^{k+\lambda}$-norm for $k\geq 2$, $E_{\sigma, \varphi}$ is elliptic. For such $\Theta$, equation \eqref{e3} yields that $\Theta \in C^{\infty}(X, \Wedge^{0,2}(T^*X \otimes \C))$ (see e.g.\cite[p.~467, Thm.~41]{besse}).  In our case, $\Theta = \overline{\mathbb G}_{\omega_0} ({\bf \Phi}(\sigma, \varphi)^{0,2})$ is zero if $\sigma=0$, hence $E$ is elliptic at $\Theta$ for $\delta_2$ sufficiently small.  We then conclude that  $\Theta=\overline{\mathbb G}_{\omega_0} ({\bf \Phi}(\sigma, \varphi)^{0,2})$, and hence also ${\bf \Phi}(\sigma, \varphi)^{0,2}$ are smooth.  Finally, by \eqref{e2}, ${\bf \Phi}(\sigma, \varphi)$ is smooth.
\end{proof}

 We next consider the Fr\'echet differentiability of the map ${\bf \Phi}(\sigma, \varphi) : U_1 \times U_2 \to  C^{k+\lambda}(X, \Wedge^2 T^*X)$.
 
 \begin{lemma}\label{l:Frechet}  In the setup of Lemma \ref{l:gualtieri}, we may shrink $\delta_1, \delta_2$ so that
 \[ {\bf \Phi}(\sigma, \varphi) : U_1 \times U_2 \to C^{k+\lambda}(X, \Wedge^2 T^*X)\]
 is a Fr\'echet differentiable map of any order $r \geq 1$.
 \end{lemma}
 \begin{proof} As in the proof of Lemma~\ref{l:smooth}, we consider
 \[  E: U_1 \times U_2 \times C^{k+2+\lambda}(X, \Wedge^{0,2}X) \to C^{k+\lambda}(X, \Wedge^{0,2}X), \qquad E(\sigma, \varphi, \Theta) := E_{\sigma, \varphi}(\Theta). \]
 As $E$ is a composition of (linear) Fr\'echet maps, it is  Fr\'echet differentiable of any given order $r$. Furthermore, as shown in the proof if Lemma~\ref{l:smooth}, the partial derivative of $E$ in the direction of $\Theta$ at $\Theta=0$ is  $2\sqrt{-1} {\bar \square}_{\omega_0}$ which is an isomorphism since $H^{0,2}(X, \C)=0$.  As $E(0,0, 0)=0$, we can apply the implicit function theorem and find a Fr\'echet differentiable map $\Theta(\sigma, \varphi)$, defined on  possibly smaller neighbourhoods $U_1, U_2$,  such that 
 \[ \Theta(\sigma, \varphi) : U_1 \times U_2 \to C^{k+2+\lambda}(X, \Wedge^{0,2}X), \qquad E(\sigma, \varphi, \Theta(\sigma, \varphi)) =0.\]
Using the uniqueness of the solution $\Theta(\sigma, \varphi)$ with the above property, we recover ${\bf \Phi}(\sigma, \varphi)$ by
 \[{\bf \Phi}(\sigma, \varphi) =  \omega_{\varphi} + 2{\rm Re}(\partial \bar \partial^*_{\omega_0} \Theta(\sigma, \varphi)) + 2{\rm Re}\left(\bar\square_{\omega_0} \Theta(\sigma, \varphi)\right).\]
 It follows from the above expression that ${\bf \Phi}$ is a differentiable Fr\'echet map, being a composition of such maps.
 \end{proof}
 \noindent

\subsection{Moser's Lemma and the map to \texorpdfstring{$\AGK^{k-2}_{\omega_0}$}{}}

Next we employ a canonical diffeomorphism pullback to the extended Gualtieri map.   For $\sigma \in \Pi$, this is the usual Moser lift defined for a family of symplectic forms.  Similar to the previous subsection, we formally extend this process to general $\sigma \in W$.

First observe that for $\sigma \in U_1 \cap \Pi$, \eqref{F/Theta} and the fact that $\bar \partial \mathbb{G}_{\omega_0}(F_{\sigma, \varphi}(t))^{0,2}=0$ (which follows from the fact that $F_{\sigma, \varphi}$ is closed) lead to
\[ F_{\sigma, \varphi}(t) = \omega_{\varphi} + d \delta_{\omega_0} \mathbb{G}_{\omega_0}\left(F_{\sigma, \varphi}(t)^{(2,0) + (0,2)}\right).\]
Now for arbitrary $(\sigma, \varphi) \in U_1 \times U_2$, we let
\[ a_{\sigma, \varphi}(t): = \delta_{\omega_0} \mathbb{G}_{\omega_0}(F_{\sigma, \varphi}(t)^{(2,0)+(0,2)})\]
which is uniquely determined by $(\sigma, \varphi)$.  Notice that by \eqref{F1-(0,2)}
\begin{equation}\label{a-0}
 a_{\sigma, \varphi}(0)=0, \qquad \dot a_{\sigma, \varphi}(0)=  \frac{1}{2}\delta_{\omega_0} \mathbb{G}_{\omega_0}(\omega_{\varphi} Q \omega_{\varphi}J),\end{equation}
where, we recall, $Q= {\rm Re}(\poiss)$.

Considering again $\sigma \in U_1 \cap \Pi$, by \eqref{poisson-scale} it follows that we get an isotopy of symplectic forms
\[ {\bf \Phi}(t\sigma, t\varphi) = F_{t\sigma, t\varphi}(1)=F_{\sigma, t\varphi}(t)= \omega_0 + t dd^c \varphi + da_{\sigma, t\varphi}(t), \qquad t\in [0,1], \]
all taming $J_0$ and interpolating between $\omega_0$ and ${\bf \Phi}(\sigma, \varphi)$.  This gives rise to a Moser lift $\Psi_{\sigma, \varphi}(t) \in {\rm Diff}_0^{k-1}(X)$, defined by integrating the time dependent vector field
\begin{equation} \label{Moser}
V_{\sigma, \varphi}(t) := -{\bf \Phi}(t\sigma, t \varphi)^{-1} \left(d^c \varphi + \frac{\partial}{\partial t}a_{\sigma, t\varphi}(t)\right), \end{equation}
which is of regularity $C^{k-1}(\R \times X)$, and which satisfies
\[ \left(\Psi_{\sigma, \varphi}(t)\right)^*({\bf \Phi}(t\sigma, t\varphi))=\omega_0.\]
Notice that since equation (\ref{a-0}) makes sense for arbitrary $\sigma$, we can use \eqref{Moser} to define, for any $(\sigma, \varphi) \in U_1 \times U_2$,  a family of diffeomorphisms $\Psi_{\sigma, \varphi}(t) \in {\rm Diff}^{k-1}_0(X)$.  For $\sigma \in \Pi$ we have $(\Psi_{\sigma, \varphi}(t))^* F_{\sigma, \varphi}(t) = \omega_0$, but not in general.

\begin{defn}\label{d:tilde-phi} In the setup of Lemma~\ref{l:gualtieri},  for any $(\sigma, \varphi) \in U_1 \times U_2$, we let 
\[\widetilde{\bf \Phi}(\sigma, \varphi) := \Psi_{\sigma, \varphi}(1) \cdot J_0\]
be the induced $K$-equivariant  map to the space of integrable almost complex structures (of regularity $C^{k-2}(X)$)  on $X$. For any $\sigma \in \Pi \cap U$, we have $\widetilde{\bf \Phi}(\sigma, \varphi) \in \mathcal{AGK}^{k-2}_{\omega_0}$.
\end{defn}
Without entering at this point to any subtle regularity analysis of $\widetilde{\bf \Phi}$, we notice that  this map is  Gateaux differentiable at $(0,0)$ in the following sense:
\begin{lemma}\label{l:Gateaux-diff} For any $(\sigma,\varphi) \in U_1 \times U_2$, 
\[ \lim_{t\to 0}\left(\frac{\widetilde{\bf \Phi}(t\sigma, t\varphi)}{t}\right) = (\mathcal{L}_{\nabla_{\omega_0} \varphi} J_0 )- \tfrac{1}{2} \mathcal{L}_{(\delta_{\omega_0} \mathbb{G}_{\omega_0}(\omega_0 Q \omega_0))^{\sharp}} J_0.\]
\end{lemma}
\begin{proof} By \eqref{poisson-scale}, $a_{s\sigma, \varphi} (t) = a_{\sigma, \varphi}(st)$, so we have
\[ 
\begin{split}
V_{s\sigma, s\varphi}(t) &= -{\bf \Phi}(ts\sigma, ts\varphi)^{-1}\left(sd^c\varphi + \frac{\partial}{\partial t} a_{s\sigma, ts\varphi}(t)\right) \\
&= -{\bf \Phi}(ts\sigma, ts\varphi)^{-1}\left(sd^c\varphi + \frac{\partial}{\partial t} a_{\sigma, ts\varphi}(st)\right)\\
&= -s\left( {\bf \Phi}(ts\sigma, ts\varphi)^{-1}\left(d^c\varphi + \left. \frac{\partial}{\partial \lambda} \right|_{\lambda =st} a_{\sigma, \lambda\varphi}(\lambda)\right)\right) \\
&= s V_{\sigma, \varphi}(st).
\end{split}\]
This yields
\[ \Psi_{s\sigma, s\varphi}(t)= \Psi_{\sigma, \varphi}(st). \]
Using the above,  we compute
\[
\begin{split}
 \lim_{t\to 0}\left(\frac{\widetilde{\bf \Phi}(t\sigma, t\varphi)}{t}\right) &=  \lim_{t\to 0}\left(\frac{\Psi_{t\sigma, t\varphi}(1) \cdot J_0}{t}\right) = \lim_{t\to 0}\left(\frac{\Psi_{\sigma, \varphi}(t) \cdot J_0}{t}\right) \\
 &= \left.\frac{d}{dt} \right|_{t=0} \left(\Psi_{\sigma, \varphi}(t) \cdot J_0\right) = -\mathcal{L}_{V_{\sigma, \varphi}(0)}J_0.
 \end{split} \]
From \eqref{Moser} (noting that ${\bf \Phi}(0,0)=\omega_0$) and  \eqref{a-0}, we get
\[V_{\sigma, \varphi}(0) = -\omega_0^{-1}(J_0 d\varphi)-  \frac{1}{2}\omega_0^{-1}\left(\delta_{\omega_0} \mathbb{G}_{\omega_0}\left(\omega_0 Q \omega_0 J_0\right)\right), \qquad Q= {\rm Re}(\sigma). \]
The lemma follows easily.
\end{proof}

\subsection{Finite dimensional GIT setup}

Inspired by the strategy of \cite{sz}, in this subsection we derive a formal finite dimensional GIT setup on $U_1 \cap W$,  where $W$ is the vector subspace generated by the Poisson locus $\Pi \subset H^0(X, \Wedge^2 T_X^{1,0})$. This uses ideas of LeBrun-Simanca~\cite{LS} to first solve 
a weaker PDE problem than the cscGK one,  which gives rise to a map from  $U_1 \cap \Pi$ to GK structures whose generalized scalar curvatures belong to a certain finite dimensional subspace of functions.  The extension of this map to the non-singular domain $U_1\cap W$ requires formally extending the generalized scalar curvature map to non-closed $F$.  We then pull back via the Moser lift to obtain a map from $U_1 \cap W$ into $\mathcal{AGK}_{\omega_0}$.  The idea is to obtain a natural symplectic structure on $U_1$ with which to run GIT arguments. Here it is crucial that the differential of this map at $0$ is complex with respect to the natural complex structures.  We are only able to establish this point in the case $\omega_0$ is K\"ahler-Einstein, through delicate K\"ahler/Bochner identities.

To begin, note that since $K$ is the group of K\"ahler isometries of $(J_0,  \omega_0, g_0)$ inside $\Aut_{\rm red}(X,J_0)$, its Lie algebra $\mathfrak{k}$ consist of all Killing vector fields $\xi$ of  $(J_0, \omega_0, g_0)$ which admit a Killing potential $h_{\omega_0}^{\xi}\in C^{\infty}(X)$, defined by the property
\[ \imath_{\xi} \omega_0 = - dh_{\omega_0}^{\xi}, \]
or, equivalently,  $\xi = J \nabla h_{\omega_0}^{\xi}$. Such Killing potentials are defined up to an additive constant, and are unique if we normalize them to have zero mean-value with respect to
 $dV_{\omega_0}$. We shall denote by $\mathfrak{k}_{\omega_0}$ the finite dimensional vector space of all normalized Killing potentials of elements of $\mathfrak{k}$ and by $\R \oplus \mathfrak{k}_{\omega_0}$ the space of all Killing potentials. It is well-known (see e.g. \cite{gauduchon-book}) that $\mathfrak{k}_{\omega_0}$  is a finite dimensional Lie algebra with respect to the $\omega_0$-Poisson bracket,  which is isomorphic to $\mathfrak{k}$.
\begin{defn}\label{d:vector-space} For $F:={\bf \Phi}(\sigma, \varphi)$, let $\mathfrak{k}_{F}$  denote the finite dimensional vector space of  $C^{k-2+\lambda}(X)$ functions, defined  by
\[ \mathfrak{k}_{F}:= \left(\Psi_{\sigma, \varphi}(1)^{-1}\right)^*(\mathfrak{k}_{\omega_0}).\]
Notice that in the above definition,  we do not require $F$ to be closed or $K$-invariant.  Further we shall denote
\[ h_{F}^{\xi}:= \left(\Psi_{\sigma, \varphi}(1)^{-1}\right)^*(h_{\omega_0}^{\xi}),\]
emphasizing again that $F$ is not in general closed or $\xi$-invariant, and  $h_{F}^{\xi}$ is not in general an $F$-Hamiltonian for $\xi$.  Finally, we denote by $\Pi_{\mathfrak{k}_{F}} : C^{k-2+\lambda}(X) \to \mathfrak{k}_{F}$ the $L^2$-orthogonal projection to $\mathfrak{k}_{F}$ with respect to $dV_F$,  and by $\Pi_{\mathfrak{k}_F}^{\perp}= \left({\rm Id} - \Pi_{\mathfrak{k}_{F}}\right)$ the $L^2$-orthogonal complement.
\end{defn}

Following \cite{ASU}, we can extend the notion of generalized scalar curvature to the space $\mathcal{AGK}_F$, i.e to the setting where $F$ is not necessarily closed.  As discussed above, when $F$ is closed $\Gscal(F, J_0)$ is identified with the momentum map at $J_0$ for the action of ${\rm Ham}_F$ on the space $\mathcal{AGK}_F$.

\begin{defn}\label{d:Gscal} Suppose $F$ is a $2$-form which tames $J$ and let  $g_F:=-(FJ)^{\rm sym}$ be the corresponding Hermitian metric on $(X, J)$, and let $b = (-FJ)^{\rm skew}$. The smooth function $\Gscal(F, J)$ defined by \eqref{eq:Gscal} is called the \emph{generalized scalar curvature} of $(F, J)$.
\end{defn}

\begin{rem}
We emphasize that this notion of $\Gscal$ agrees with the scalar curvature $\Goto$ for genuine generalized K\"ahler structures, but not in general.  Whereas Goto's scalar curvature is extended to nonintegrable structures by the formal moment map property, here we give an ad-hoc extension to obtain an explicit Frechet differentiable map on pairs $(F, J)$.
\end{rem}

We next adopt ideas of \cite{LS, sz} to construct the required map on the level of $2$-forms.

\begin{lemma}\label{l:sz} There exists a $K$-equivariant infinitely differentiable map $\boldsymbol{S} : U_1 \to C^{k+\lambda}(X, \Wedge^2 T^*X)$ such that 
\begin{enumerate}
\item $\boldsymbol{S}(0)=\omega_0$;
\item $\boldsymbol{S}(\sigma)$ tames $J_0$ and $\Gscal(\boldsymbol{S}(\sigma), J_0) \in \R \oplus \mathfrak{k}_{\boldsymbol{S}(\sigma)}$.
\item If $\sigma \in \Pi$, then $(\boldsymbol{S}(\sigma), J_0)$ is generalized K\"ahler.
\end{enumerate}
\end{lemma}

\begin{proof} We compose the map ${\bf \Phi}$ with $\Gscal$ to yield
\[ {\boldsymbol{\Psi}}(\sigma, \varphi) := \Gscal(\boldsymbol{\Phi}(\sigma, \omega_0 + dd^c_{J_0} \varphi), J_0).\]
This is a  Fr\'echet differentiable map (as a composition of such maps, see Lemma~\ref{l:Frechet}) defined on  $U_1 \times U_2$ with values in  $C^{k-2+\lambda}(X)$. 
Observing that ${\boldsymbol{\Psi}}(0, \varphi)= \Scal(\omega_\varphi)$, and using that $\omega_0$ is cscK, it follows that the differential of ${\bf \Psi}$ in the direction of $U_2$ at $\varphi=0$ is therefore
\[ (d_{U_2} {\boldsymbol{\Psi}})_{0} (\dot \varphi)= \mathbb{L}_{\omega_0}(\dot \varphi), \]
where $\mathbb{L}_{\omega_0}=\delta_{\omega_0} \delta_{\omega_0} \nabla_{\omega_0}^{-} d$ is the Lichnerowicz operator of $(M, J_0, \omega_0, g_0)$.  Recall that the kernel of $\mathbb{L}_{\omega_0}$ is precisely the space $\R \oplus \mathfrak{k}_{\omega_0}$. We can thus consider the LeBrun--Simanca~\cite{LS} modified map
\[ {\boldsymbol{\Psi}}^{\perp} : U_{1} \times U_2^{\perp} \to (C^{k-2+\lambda}(X))^{\perp}, \qquad {\boldsymbol{\Psi}}(\sigma, \varphi)^{\perp}:= \Pi_{\mathfrak{k}_{\omega_0}}^{\perp}\left(\Gscal(\boldsymbol{\Phi}(\sigma, \varphi), J_0)\right),\]
where $(\cdot )^{\perp}$ stands for the $L^2(X, \omega_0)$-orthogonal  complement of $\R \oplus \mathfrak{k}_{\omega_0}$ and $\Pi_{\mathfrak{k}_{\omega_0}}^{\perp}$ for the $L^2(X, \omega_0)$-orthogonal projection to $(\cdot )^{\perp}$.
We can apply the implicit function theorem to $\boldsymbol{\Psi}^{\perp}$ in this setup,  and obtain an $r$-differentiable map 
\begin{equation}\label{eq:phi-sigma} \sigma \to \varphi_\sigma \in C^{k+2+\lambda}(X)^{\perp} \end{equation}
such that  
\[ \varphi_0\equiv 0,  \, \, \qquad \, \,  \Gscal(\boldsymbol{\Phi}(\sigma, \varphi_\sigma), J_0) \in \R \oplus \mathfrak{k}_{\omega_0} \quad \forall \sigma \in U_1.\]
We now define
\[ \boldsymbol{S}(\sigma):= \boldsymbol{\Phi}(\sigma, \varphi_\sigma).\]
This map is Fr\'echet differentiable of order $r$ by Lemma~\ref{l:Frechet}.  Using that the vector spaces $\mathfrak{k}_{F_{\sigma}}$ in Definition~\ref{d:vector-space} continuously vary in $\sigma$, we apply the now standard argument from \cite{LS} to conclude that $\Gscal(\boldsymbol{S}(\sigma), J_0) \in \mathfrak{k}_{F_{\sigma}}$.
\end{proof}

\begin{cor} Let 
\begin{align*}
  \widetilde{\bf S}(\sigma):= \widetilde{\boldsymbol{\Phi}}(\sigma, \varphi_{\sigma}),  
\end{align*}
where $\varphi_{\poiss}$ is given by \eqref{eq:phi-sigma} and $\widetilde{\boldsymbol{\Phi}}$ is the $2$-form introduced in Definition~\ref{d:tilde-phi}.
For any $r>1$, the integer $k$ in the definitions of the maps ${\bf \Phi}$ and ${\bf S}$ can be chosen large enough so that $\widetilde{\bf S}$ is $r$-times continuously differentiable, $K$-equivariant, and satisfies \begin{enumerate}
\item $\widetilde{\bf S}(0)= J_0$.
\item Shrinking $\delta_1$ if necessary, $\widetilde{\bf S}$ takes values in $\mathcal {AGK}_{\omega_0}$.
\item $\Gscal(\omega_0, \widetilde{\bf S}(\sigma)) \in \R \oplus \mathfrak{k}_{\omega_0}$ for all $\sigma \in U_1 \cap \Pi$.
\end{enumerate}
\end{cor}
\begin{proof} Item (1) is clear by construction, hence item (2) follows by continuity.  By construction, ${\bf S}(\poiss)= {\bf \Phi}(\poiss, \varphi_{\poiss})$ is a $2$-form of regularity $C^{k+\lambda}(X)$, with regularity $C^r$ in $\poiss$. We will assume that $k \geq r+2$. Then, by Definition~\ref{d:tilde-phi}, $\widetilde {\bf S} = \widetilde {\bf \Psi}(\poiss, \varphi_{\poiss})$ will define an integrable almost complex structure $\AGK_{\omega_0}^r$ of regularity $C^r$  with respect to $\poiss$. The $K$-equivariance of $\widetilde{\bf S}$ follows by the naturality of the construction whereas  the last property  follows by the definition of ${\bf S}(\poiss)$.
\end{proof}

Now consider the pullback by $\widetilde{\bf S}$ of the formal symplectic structure $\boldsymbol{\Omega}$ on $\AGK_{\omega_0}$ to $U_1\subset W$  (which we still denote $\boldsymbol{\Omega}$). 
\begin{lemma} Suppose $(X, \omega_0, J_0)$ is a K\"ahler-Einstein Fano manifold.  Then the differential  of $\widetilde{\bf S}$ at $0$ is non-degenerate and complex with respect to the natural complex structure of $U_1 \subset W$ and the formal complex structure ${\bf J}$ on $\AGK_{\omega_0}$. In particular,  $\boldsymbol{\Omega}$ defines a $K$-invariant symplectic structure which tames the natural complex structure on a small neighbourhood  $U_1 \subset W$.
\end{lemma}
\begin{proof} 
As $W$ is the span of $\Pi$, it is enough to check the statement for the directions generated by $\sigma \in \Pi$. Thus, let $0\neq \sigma \in U_1 \cap \Pi$ and $\varphi_{\sigma}$ denote the corresponding function such that ${\widetilde{\bf S}}(\sigma)=\widetilde{\bf \Phi}(\sigma, \varphi_{\sigma})$. We consider $(t\sigma, \varphi_{t\sigma})$ and denote by $\dot \varphi_{\sigma} = \left. \frac{d}{dt} \right|_{t=0} \varphi_{t\sigma}$. By Lemma~\ref{l:Gateaux-diff}, we have
\begin{equation}\label{S-diff}
\begin{split}
(D_0 \widetilde{\bf S})(\sigma) &= \left. \frac{d}{dt} \right|_{t=0} \widetilde{\bf \Phi}(t\sigma, \varphi_{t\sigma}) 
= \mathcal{L}_{\nabla \dot\varphi_{\sigma}}J -\tfrac{1}{2} \mathcal{L}_{ \left(\delta \mathbb{G}(\omega_0 Q \omega_0) \right)^{\sharp}} J,  \end{split}\end{equation}
where  $Q={\rm Re}(\sigma)$ and $\nabla, \delta, \mathbb{G}$ are the Levi-Civita connection, co-differential and Green operator of $(g_0, \omega_0, J_0)$. We use the notation $\delta$ for the action of $\nabla^*$ on tensors.  We also let $J=J_0$ to lighten the notation.

We first notice that $(D_0 \widetilde{\bf S})(\sigma) \neq 0$ if $\sigma \neq 0$. Indeed, if this was not true, then the $1$-form \[ \frac{1}{2} \delta \mathbb{G} (\omega_0 Q\omega_0) - d\dot \varphi_{\sigma} \] 
would be  the Riemannian dual of a (real) holomorphic vector field.  As this $1$-form is $L^2$-orthogonal to harmonic $1$-forms, it follows that  (see \cite[Ch.2]{gauduchon-book}) 
\[\tfrac{1}{2} \delta \mathbb{G} (\omega_0 Q\omega_0)= dh + d^c f, \]
for some smooth functions  $f, h$. Applying $\delta$ and $\delta^c$ to the above equality (and using that $J^{*} \delta \eta = \delta (\eta J)$ for $\eta$ a $(2,0)+(0,2)$ form), we obtain $\Delta h = \Delta^c f=0$, so $f$ and $h$ must be constant functions and thus
\[  \delta \mathbb{G} (\omega_0 Q\omega_0)=0=\delta \mathbb{G}(\omega_0 JQ \omega_0).\]
It follows that the $\bar \partial$-closed $(0,2)$-form $\omega_0 \sigma \omega_0$ is harmonic and thus zero by the assumption $H^{2,0}(X, \C)=H^{0,2}(X, \C)=0$.  This contradicts our assumption that $\sigma \neq 0$.

Using the momentum map interpretation of $\Gscal$ (see e.g. \cite[Thm.3.8]{ASU} or \cite{Goto-moment})  we get \[\left. \frac{d}{dt} \right|_{t=0} \Gscal(\omega_0, \tilde{\bf S}(t\poiss))= \delta \delta ((D_0\tilde{\bf S})(\poiss))=\delta \delta  \left(J\mathcal{L}_{\nabla \dot \varphi_{\sigma}} J  - \tfrac{1}{2}J\mathcal{L}_{\delta \mathbb{G}(\omega_0 Q \omega_0)}J\right).\]
On the other hand, as $\Gscal(\omega_0, \tilde{\bf S}(t\poiss))\in  \R \oplus \mathfrak{k}_{\omega_0}$, we have $\left. \frac{d}{dt} \right|_{t=0} \Gscal(\tilde{\bf S}(t\poiss)) \in \R \oplus \mathfrak{k}_{\omega_0}$  and thus
\[\int_X \left(\left. \frac{d}{dt} \right|_{t=0} \Gscal(\omega_0, \tilde{\bf S}(t\poiss)) \right)^2\omega_0^{[n]}=\int_X \left \langle g_0^{-1}\nabla^{-} d \left(\frac{d}{dt}_{|_{t=0}}\Gscal(\omega_0,\tilde{\bf S}(t\poiss)) \right), (D_0\tilde{\bf S})(\poiss)\right\rangle_{g_0}=0,  \]
where for any $1$-form $\eta$, $\nabla^{-} \eta$ stands for $(2,0)+(0,2)$ part  with respect to $J_0$ of the $(2,0)$-tensor $\nabla \eta$.
We thus conclude
\[\left. \frac{d}{dt} \right|_{t=0} \Gscal(\omega_0, \tilde{\bf S}(t\poiss))=\delta \delta  \left(J\mathcal{L}_{\nabla \dot \varphi_{\sigma}} J  - \frac{1}{2}J\mathcal{L}_{\delta \mathbb{G}(\omega_0 Q \omega_0)}J\right)=0. \]
Using the identity
\[ \mathcal{L}_X J = - [\nabla X, J] \]
the above can be re-written as
\begin{equation} \label{key} \delta \delta \left( \nabla^{-} d\dot \varphi_{\sigma}  - \tfrac{1}{2} \nabla^{-} \delta \mathbb{G}(\omega_0 Q \omega_0) \right)=0.\end{equation}
By Bochner's identity (see \cite[Lemma~1.23.5]{gauduchon-book}), for any $1$-form $\eta$
\[ \delta \delta \nabla^{-} \eta = \tfrac{1}{2} \Delta \delta \eta  - \langle d^c\eta, \rho_{\omega_0}\rangle + \tfrac{1}{2}\langle\eta, d \Scal(\omega_0)\rangle. \]
Using this identity for $\eta = \delta \mathbb{G}(\omega_0Q\omega_0)$ and the fact that $\Scal(\omega_0)$ is constant, we obtain
\[ \delta \delta \nabla^{-} \delta\mathbb{G}(\omega_0Q \omega_{0}) = -\langle d^c\delta \mathbb{G}(\omega_0Q \omega_0), \rho_{\omega_0}\rangle. \]
In the K\"ahler--Einstein case, using the K\"ahler identity $[\Lambda_{\omega_0}, d^c]=\delta$, we get
\begin{equation}\label{KE-case}
\delta \delta \nabla^{-} \delta\mathbb{G}(\omega_0Q \omega_{0}) = -\left(\frac{\Scal(\omega_0)}{2n}\right) \Lambda_{\omega_0}\left(d^c\delta \mathbb{G}(\omega_0Q \omega_0)\right)=0.
\end{equation}
We then conclude by \eqref{key} and \eqref{KE-case} that $\dot \varphi_{\sigma}$ is  in the kernel of the Lichnerowicz operator $\delta \delta \nabla^{-} d$, i.e. $g_0^{-1}\nabla^{-}d \dot\varphi_{\poiss}=  J\mathcal{L}_{\nabla \dot \varphi_{\sigma} }J=0$.  Substituting back to \eqref{S-diff}, we obtain
\[ (D_0\widetilde{\bf S})(Q - \i J Q) = -\tfrac{1}{2} \mathcal{L}_{\delta \mathbb{G}(\omega_0 Q \omega_0)} J,\]
which is manifestly ${\bf J}$-complex. 
\end{proof}

\begin{lemma}\label{l:induced-moment-map} The action of $K$ on $(U_1, \boldsymbol{\Omega})$ is Hamiltonian with momentum map $\boldsymbol{\mu}(\sigma) \in \mathfrak{k}^*={\rm Lie}(K)^*$ given by
\[ \langle \boldsymbol{\mu}(\sigma), \xi \rangle = - \int_X \mathring{\Goto}(\omega_0,\widetilde{\bf S}(\sigma))  h_{\omega_0}^\xi \omega_0^{[n]}, \]
where $\mathring\Goto(\omega_0, J)$ is the normalized function corresponding to the momentum map for the Hamiltonian action of $\Ham_{\omega_0}$ on $\AGK_{\omega_0}$ defined in \cite{Goto-Math.Ann},  and  $h_{\omega_0}^{\xi} \in \mathfrak{k}_{\omega_0}$ denotes the normalized Hamiltonian of an element $\xi \in \mathfrak{k}$.
In particular, $\boldsymbol{\mu}(\sigma)=0$ for $\sigma \in U_1 \cap \Pi$ if and only if $\Gscal(\omega_0, \widetilde{\bf S}(\sigma))$ is constant and $\sigma \in U_1 \cap \Pi$ is a critical point of $||\boldsymbol{\mu}(\sigma)^{\sharp}||^2:= \int_X \left( \mathring{\Goto}(\omega_0, \tilde{\bf S}(\poiss))\right)^2\omega_0^{[n]}$ if and only if $\Gscal(\omega_0, \widetilde{\bf S}(\sigma))$ is a Killing potential for $(\omega_0, \widetilde{\bf S}(\sigma))$.
\end{lemma}
\begin{proof} Uses the momentum map property of $\mathring{\Goto}(\omega_0, J)$ and the fact (established in \cite{ASU})  that for $\sigma \in \Pi$, we have
$\mathring{\Goto}(\omega_0, \tilde{\bf S}(\poiss))=\mathring{\Gscal}(\omega_0, \widetilde{\bf S}(\poiss)) \in \mathfrak{k}_{\omega_0}$.
\end{proof}

\subsection{Proof of Theorem~\ref{thm:main}}

\begin{proof}[Proof of Theorem~\ref{thm:main}] The first part follows by Proposition~\ref{p:Kahler/Poisson-K-polystable}, noting that the K-polystability of $X$ is equivalent to the existence of a K\"ahler-Einstein metric in $2\pi c_1(X)$ by the Yau-Tian-Donaldson correspondence, and hence $\Aut_{\rm red}(X)=\Aut_{\circ}(X)$ is reductive by the Matsushima theorem.

For the second part, we assume that $\poiss\neq 0$ is a polystable Poisson structure and want to show  that there exists a smooth cscGK metric in $\GK_{\lambda\poiss, 2\pi c_1(X)}$ for $|\lambda|$ small enough, again using that 
$X$ admits a K\"ahler-Einstein metric $\omega_0 \in 2 \pi c_1(X)$.
Lemma \ref{l:induced-moment-map} defines a tamed moment map setup on $(U_1, \boldsymbol{\Omega}, K)$.  We further restrict to the open ball $U_0$ guaranteed by Proposition~\ref{p:local-tamed-GIT}.
Multiplying $\poiss$ by $\lambda\neq 0$, we get $\lambda\poiss \in U_0$ (which is also polystable).  We drop the factor of $\lambda$ and simply now consider the problem for polystable $\poiss \in U_0$.  By Proposition~\ref{p:local-tamed-GIT},  there exists a Poisson structure $\sigma_{\infty} \in K^{\C}\cdot \sigma \cap U_0$ such that   $\mu^K_{\omega}(\sigma_{\infty})^{\sharp}$ is in the Lie algebra of the stabilizer of $\sigma_{\infty}$ in $K$. By Lemma~\ref{l:induced-moment-map}, this is equivalent to $\boldsymbol{\mu}(\tilde {\bf S}(\poiss_{\infty}))^{\sharp} = -(\mathring{\Gscal} (\omega_0, \tilde{\bf S}(\sigma_{\infty}))\in \mathfrak{k}_{\omega_0}$  being a Hamiltonian potential of a vector field which preserves the GK structure $(\omega_0, \tilde{\bf S}(\sigma_{\infty})) \in \AGK^r_{\omega_0}$, i.e. we get an extremal GK structure in the sense of \cite{ASU}.  Denote by $\T_{\rm ext}<K$ the torus generated by the Killing potential $\mathring{\Gscal} (\omega_0, \tilde{\bf S}(\sigma_{\infty}))$, and observe that $\tilde{\bf S}(0)=J_0$ and $\tilde{\bf S}(\sigma_{\infty})$ are both $\T_{\rm ext}$-invariant.  It is shown in \cite{ASU} that the $L^2(X, \omega_0)$-projection of $\mathring{\Gscal}( \omega_0,J)$ to the finite dimensional space of $\omega_0$-Hamiltonian functions  of the elements  of ${\rm Lie}(\T_{\rm ext})$  is independent of the choice of $\T_{\rm ext}$-invariant GK structure $(J, \omega_0)$ in $\mathcal{AGK}^r_{\omega_0}$: it thus follows that both $(\omega_0, \tilde{\bf S}(0))$ and $(\omega_0, \tilde {\bf S}(\sigma_{\infty}))$ have constant generalized scalar curvature.  The cscGK structure $(\omega_0, \tilde{\bf S}(\poiss_{\infty})) \in \AGK^r_{\omega_0}$ is of regularity $C^r$, hence we apply Theorem~\ref{t:schauder} to obtain that it is actually smooth.  We have thus found a smooth cscGK metric ${\bf S}(\poiss_{\infty}) \in \GK_{\poiss_{\infty}, 2\pi c_1(X)}$ on $X$ (see Lemma~\ref{l:sz}). As $\sigma_{\infty} \in K^{\C} \cdot \sigma$, 
pulling everything by  an element of $K^{\C}$ produces a smooth cscGK metric in $\GK_{\poiss, 2\pi c_1(X)}$.
\end{proof}


\providecommand{\bysame}{\leavevmode\hbox to3em{\hrulefill}\thinspace}
\providecommand{\MR}{\relax\ifhmode\unskip\space\fi MR }
\providecommand{\MRhref}[2]{%
  \href{http://www.ams.org/mathscinet-getitem?mr=#1}{#2}
}
\providecommand{\href}[2]{#2}

\end{document}